\documentclass{amsart}

\usepackage{graphicx}
\usepackage{subcaption}
\graphicspath{{images/}}
\usepackage{amsmath,amsthm,amssymb,amsfonts, bm, bbm, stmaryrd}
\usepackage{pgfplots}
\usepackage{ben_macro}

\usepackage[utf8]{inputenc}
\usepackage[T1]{fontenc}
\usepackage{fullpage}
\usepackage{tikz}
\usetikzlibrary{arrows.meta, bending}

\pgfplotsset{compat=1.18}

\begin{document}

\title{A nodally bound-preserving finite element method for hyperbolic 
convection-reaction problems}

\author[1,*]{Ben S. Ashby}

\author[1]{Abdalaziz Hamdan}

\author[1,2]{Tristan Pryer}

\address{$^1$ Institute for Mathematical Innovation\\ University of
  Bath, Bath, UK. \\$^2$ Department of Mathematical Sciences \\
  University of Bath, Bath, UK.}

\begin{abstract}
  In this article, we present a numerical approach to ensure the
  preservation of physical bounds on the solutions to linear and
  nonlinear hyperbolic convection-reaction problems at the discrete
  level. We provide a rigorous framework for error analysis,
  formulating the discrete problem as a variational inequality and
  demonstrate optimal convergence rates in a natural norm. We
  summarise extensive numerical experiments validating the
  effectiveness of the proposed methods in preserving physical bounds
  and preventing unphysical oscillations, even in challenging
  scenarios involving highly nonlinear reaction terms. 
\end{abstract}

\maketitle

\section{Introduction}

Ensuring that numerical approximations respect certain physical bounds
presents a significant challenge across various fields in
computational modelling. Many important PDE models rely on solutions
that adhere to these bounds to maintain physical validity. For
instance, in phase-field modelling, solutions must strictly preserve
global maxima and minima \cite{wang2022unconditionally}, while in
incompressible flow simulations, the velocity field is required to
remain divergence-free \cite{schroeder2018divergence}. Similarly, in
applications such as kinetic equations, dose deposition modelling for
external beam radiotherapy
\cite{AshbyChronholmHajnalLukyanovMacKenziePimPryer:2024}, or nuclear
engineering \cite{calloo2024cycle}, maintaining physical accuracy is
critical to ensure reliable and safe operations, particularly when
computational constraints limit the achievable numerical
fidelity. These bounds, derived from fundamental physical principles,
become especially important when the solution serves as input to
subsequent models, where unphysical values could propagate errors and
lead to erroneous predictions downstream.

A common manifestation of the failure of numerical solutions to
satisfy these bounds is the oscillation of the solution around layers
or irregularities in the PDE solution. For example, singularly
perturbed elliptic problems often exhibit boundary layers, while
hyperbolic problems can support discontinuities, which can lead to
under- and over-shoots numerically (see \S\ref{sec:numerics}). Even
many stabilised numerical methods can exhibit unphysical oscillations
that overshoot these maxima and minima, as illustrated in Figure
\ref{fig:ex_2_x_sections}.

The aim of the method presented in this work is to prevent spurious
numerical behaviour from violating maximum principles. We propose a
finite element method that enforces physical bounds at nodal points by
discretising the problem as a variational inequality. This ensures
that the solution lies within a finite-dimensional, closed and convex
subset of the natural solution space for the PDE. We provide bounds on
the approximation error of the proposed bound-preserving finite
element method, demonstrating optimal convergence rates for a scalar
hyperbolic equation. This formulation enables rigorous error analysis
and yields a simple implementation whilst also offering flexibility in
computational mesh selection. Our analysis focusses on
convection-reaction equations, a class of problems where bound
preservation is particularly important due to the challenges posed by
sharp gradients, boundary layers and discontinuities. Moreover, we
extend our approach to nonlinear problems, leveraging appropriate
quasi-norms to establish best-approximation results within the
framework of variational formulations.

Standard finite element methods often fail to respect physical bounds
without imposing additional constraints or modifications. For example,
the most popular finite element methods for incompressible flows often
fail to ensure pointwise divergence-free solutions, and proving energy
stability in dissipative systems is similarly challenging (see, e.g.,
\cite{GiesselmannMakridakisPryer:2014,CelledoniJackaman:2021}). This
was formalised early on in \cite{CR73}, which demonstrated that
piecewise linear finite elements respect such bounds only under
specific mesh conditions.

A variety of methods have been proposed to address these challenges. A
notable approach is the development of methods satisfying the
\emph{Discrete Maximum Principle (DMP)}, particularly for
convection-dominated convection-diffusion problems (see
\cite{MH85,XZ99,BE05,Kuz07,BJK17,BJK23,wangdiscrete}, among
others). These methods often involve nonlinear stabilisation, where
additional nonlinear terms are introduced to ensure that the discrete
solution does not violate the maximum principle. Despite their
effectiveness, these methods frequently rely on piecewise linear
elements and extending them to higher-order elements introduces
stricter mesh conditions and additional analytical complexities.

Related approaches include nodally bound-preserving stabilised methods
\cite{barrenechea2024nodally,amiri2024nodally,BarrenecheaPryerTrenam:2024},
where the discrete solution is shown to satisfy a variational
inequality. The idea of enforcing bounds by formulating the problem as
a variational inequality has also appeared in
\cite{chang2017variational,kirby2024high,keith2024proximal}. These
methods often involve solving nonlinear problems, a necessity
highlighted by the Godunov barrier theorem, which asserts that
achieving high-order accuracy while preserving monotonicity or bounds
is generally impossible with linear methods. Consequently,
nonlinearity is an inherent feature of any method aiming to respect
maximum principles whilst retaining higher-order accuracy.

The rest of the paper is structured as follows: In
\S\ref{sec:problem_setup} we introduce a model hyperbolic problem and
discuss its solution, regularity and approximation by the finite
element method. We then conduct an error analysis in
\S\ref{sec:linear_analysis}. Following in \S\ref{sec:nonlinear}, we
provide a best approximation result for a hyperbolic problem with
singular nonlinearity in the reaction term. Implementation details
including methods to obtain an approximation to the finite element
solution are given in \S\ref{sec:iterative_methods}. Finally,
numerical experiments are presented in \S \ref{sec:numerics}.
 
\section{Problem formulation \& discretisation}\label{sec:problem_setup}

In what follows, $\W \subseteq \reals^d$, $d = 2, 3$, is assumed to be
a Lipschitz domain, ensuring that a unit outward normal vector
$\vec{n}$ is defined almost everywhere on $\partial \W$. For any
measurable subset $\omega \subseteq \W$, let $\leb{p}(\omega)$ denote
the space of $p$-th power (Lebesgue) integrable functions over
$\omega$, equipped with the norm $\Norm{\cdot}_{\leb{p}(\omega)}$. The
$\leb{2}$ inner product over $\omega$ is written as
$\ltwop{u}{v}_{\omega} := \int_{\omega} uv \diff x$, with the
convention that the subscript is omitted if $\omega = \W$.

The standard Sobolev space $\sob{k}{p}(\omega)$ is defined as the
space of functions in $\leb{p}(\omega)$ whose weak derivatives of
order at most $k$ are also in $\leb{p}(\omega)$. As is customary,
$\sob{k}{2}(\omega)$ is denoted by $\sobh{k}(\omega)$.

For $\vec b : \W \to \reals^d$ and $c : \W \to \reals$, we introduce
the linear hyperbolic problem given by
\begin{equation}\label{eq:pde_classical_form}
  \begin{split}
  \vec b \cdot \grad u + c u 
  &= 
  f \quad \text{in}\,\,\W\\
  u &= g\quad \text{on}\,\,\Gamma_-,
  \end{split}
\end{equation}
where $\Gamma_-$ is the inflow boundary defined by
the flow field $\vec b$, that is,
\begin{equation*}
  \Gamma_- 
  =
  \{\vec x \in \partial \W : \vec b(\vec x) \cdot \vec n(\vec x) < 0\}, 
\end{equation*}
and
\begin{equation}
  \Gamma_+ = \partial\W \setminus \Gamma_-.
\end{equation}
The analysis of this linear problem is significantly more complicated
than for advection-diffusion-reaction problems, as the smoothing
properties enjoyed by elliptic operators are not present here. For
example, even when $\vec b$ is smooth, discontinuities can propagate
along streamlines, leading to solutions which are quite irregular. We
elaborate further on this point in Remark
\ref{rem:regularity_remarks}.

We now proceed to discuss questions of existence, uniqueness and
regularity of solutions to \eqref{eq:pde_classical_form}. Let $\vec b
\in C^1(\bar{\W})$ and $ c \in C^0(\bar{\W})$ and suppose there exists
$\mu > 0$ such that
\begin{equation}\label{eq:coercivity_assumption}
  c(x) - \frac 1 2 \grad \cdot \vec b(\vec x) \geq \mu \text{ for almost every } \vec x \in \W.
\end{equation}
We define the function space
\begin{equation}\label{eq:graph_space}
  \operatorname{H}_-(\W)
  = 
  \{v \in \leb{2}(\W) : \vec b \cdot \grad v + c v \in \leb{2}(\W), \,\,
  (\vec b \cdot \vec n) v = 0 \,\,\text{on}\,\, \Gamma_-\},
\end{equation}
where boundary values are interpreted in the trace sense. We refer to
\cite[\S3.1]{scott2022transport} for a justification of sufficient
boundary regularity so that the boundary condition included in this
definition is meaningful. The space $\operatorname{H}_-(\W)$ is a
Hilbert space when equipped with norm
\begin{equation}
  \Norm{v}^2_{\operatorname{H}_-(\W)}
  =
  \Norm{v}^2_{\leb{2}(\W)}
  +
  \Norm{\vec b \cdot \grad v + c v}^2_{\leb{2}(\W)},
\end{equation}
allowing us to give the variational formulation of problem
\eqref{eq:pde_classical_form}: find $u \in \operatorname{H}_-(\W)$ such that
\begin{equation}\label{eq:pde_variational_form}
  a(u, v)
  =
  l(v) \quad \forall v \in \leb{2}(\W),
\end{equation}
where 
\begin{equation}\label{eq:a_bilinear}
  a(w, v)
  =
  \int_{\W} \left( \vec b \cdot \grad w + c w\right) v,
\end{equation}
and 
\begin{equation}\label{eq:l_linear}
  l(v) = \int_{\W}fv.
\end{equation}

The following result derived from Green's formula will be frequently used, and is stated
here for convenience.
\begin{equation}\label{eq:green_formula}
  \int_{\W}\left(\vec b \cdot \grad v\right)w 
  =
  \int_{\partial \W} vw (\vec b \cdot \vec n)
  -
  \int_{\W}v \left(\vec b \cdot \grad w\right) 
  -
  \int_{\W}vw \operatorname{div}(\vec b).
\end{equation}

We now present the main existence result, which assumes \lq good' behaviour of
the advection field $\vec b$

\begin{proposition}[Existence \& uniqueness of solutions \cite{scott2022transport}]
  Suppose that $\W$ is a Lipschitz domain, and assume that the vector field $\vec b$
  lies in $C^1(\bar{\W})$, with all vector components strictly positive. We additionally
  assume $c \in C(\bar{\W})$. We assume additionally that $g$ can be extended to $\W$
  such that the result lies in $\operatorname{H}_-(\W)$. For $f \in \leb{2}(\W)$, there
  exists a unique solution $u \in \operatorname{H}_-(\W)$ to problem
  \eqref{eq:pde_variational_form}.
\end{proposition}

\begin{remark}[Further remarks on regularity]\label{rem:regularity_remarks}
  The conditions for the solution to problem
  \eqref{eq:pde_variational_form} to have full
  $\sobh{1}(\W)$-regularity are strong, requiring $f \in
  \sobh{1}_0(\W)$ in addition to higher regularity on $\vec{b}$ (see
  \cite{rauch1972l2}). However, in special cases, less regular
  solutions can still be shown to exist. For example, as discussed in
  \cite[Part III, 1.1]{stynes_robust}, in the special case where $\W$
  is the unit square and the components of $\vec{b}$ are bounded away
  from zero, there is a unique solution that lies in $\operatorname{H}_-(\W)$ even
  when $g \in \leb{2}(\Gamma_-)$, meaning problem
  \eqref{eq:pde_classical_form} supports discontinuous solutions.

  More generally, as shown in \cite[Theorem 3.1]{stynes_robust}, the
  problem admits a unique solution in $\operatorname{H}_-(\W)$ provided $g$ has an
  extension in $\operatorname{H}_-(\W) \cap \leb{q}(\W)$ for some $q \geq 2$. Similar
  results on the regularity of solutions can be found in
  \cite{scott2022transport,bernard2012steady,girault2010regularite}.

  Of course, the numerical analysis of such cases is complicated by
  this lack of regularity. Indeed, it is precisely in these
  scenarios-where the boundary data or the solution exhibits low
  regularity-that spurious behaviour is frequently observed upon
  discretisation with many numerical methods.
\end{remark}

\begin{remark}[Problems satisfying a priori bounds]
For some choices of $\vec b$, problem \eqref{eq:pde_classical_form}
can be solved analytically using the method of characteristics. One
such case is when $\vec b$ has well-defined, non-intersecting
streamlines with no stationary points. These conditions ensure that
any $\vec x \in \W$ lies on a single characteristic curve which
originates at the inflow boundary, and therefore $u$ is the solution
of an ODE along this characteristic with initial data given by the
inflow boundary condition. Assuming further that $c \geq 0$ and that
the boundary data $g \in \leb{\infty}(\Gamma_-)$, we have for example
that if $f = 0$ almost everywhere,
\begin{equation}
  u(\vec x)
  \leq
  \Norm{g}_{\leb{\infty}(\Gamma_-)}.
\end{equation}
\end{remark}

\subsection{Finite element discretisation}

Let $u$ be the solution of \eqref{eq:pde_variational_form} and,
without loss of generality, we assume that $u$ satisfies the bound
$u(\vec x) \in [0,1]$ for almost every $\vec x \in \W$.  We assume
that the domain $\W$ is subdivided into a conforming, shape-regular
triangulation $\T$, namely, $\T$ is a finite family of sets such that
\begin{enumerate}
  \item $K\in\T{}$ implies $K$ is an open simplex or box,
  \item for any $K,J\in\T{}$ we have that $\overline K\cap\overline J$ is a full
    lower-dimensional simplex (i.e., it is either $\emptyset$, a vertex,
    an edge or the whole of $\overline K$ and $\overline J$) of both
    $\overline K$ and $\overline J$ and
  \item $\bigcup_{K\in\T{}}\overline K=\overline\W$.
  \end{enumerate}
We additionally assume that the elements align with transitions
between inflow and outflow boundaries. The finite element space $\fes$
is defined to be
\begin{equation}
  \fes
  :=
  \{v_h \in C^0(\W) : v_h\vert_K \in \mathcal{R}(K)\,\, \forall K \in \T,\,
  v_h =0 \text{ on } \Gamma_-\},
\end{equation}
where $\mathcal{R}(K)$ is either $\mathbb{P}_k(K)$, the space of
polynomials of degree $k$, or $\mathbb{Q}_k(K)$, the space of
polynomials of total degree $k$, depending upon the triangulation. For
an element $K \in \T$, we denote the diameter of $K$ by $h_K$, with $h
:= \max_{K \in \T} h_k$.  Finally, let $\vec x_1, ..., \vec x_N$ be
the union of the set of interior degrees of freedom and those that lie
on the outflow boundary.

We define the convex subset $K_h$ of
$\fes$ by restricting the nodal values of function in $\fes$, that is 
\begin{equation}\label{eq:definition_of_K_h}
  K_h 
  := 
  \{v_h \in \fes : v_h(\vec x_i) \in [0, 1],\,\,i=1,\dots,N\}.
\end{equation}

\begin{remark}[Preservation of the bound at the degrees of freedom]
  For polynomial degree, $k=1$, the space $K_h$ consists of precisely
  the finite element functions which satisfy the upper and lower bound
  pointwise since between nodes the functions are (bi)linearly
  interpolated. For polynomial degree, $k=2$, or higher, this is not
  the case, and $K_h$ is the set of finite element functions which are
  \emph{nodally} bound preserving. This property was investigated in
  \cite{barrenechea2024nodally} for second order elliptic problems,
  where the authors obtain their approximations via a nonlinear
  stabilised method rather than a variational inequality. There, the
  bound preserving approximation was shown to be equivalent to the
  solution of a discrete variational inequality.
\end{remark}
Let $\delta_K$, $K \in \T$, be positive real numbers and let $a_h$
denote the bilinear form associated with the SUPG stabilisation method
\cite{brooks1982streamline, johnson1984finite}, defined as
\begin{equation}\label{eq:supg_bilinear_form}
  a_h(w_h,v_h)
  =
  \int_{\W}\left(\vec b \cdot \grad w_h + c v_h\right)v_h 
  +
  \sum_{K\in \T}\delta_K
  \int_{K}\left(\vec b \cdot \grad w_h + c w_h\right)\vec b \cdot \grad v_h.
\end{equation}
The numbers $\delta_K$ are the SUPG parameters which determine the
local degree of stabilisation. The bilinear form $a_h$ has an
associated norm for $w\in \operatorname{H}_-$
\begin{equation}\label{eq:triplenorm}
  \tripleNorm{w}^2
  :=
  \mu\Norm{w}^2_{\leb{2}(\W)} 
  +
  \sum_{K\in \T}\Norm{\delta_K^{\frac 1 2}\vec b \cdot \grad w}^2_{\leb{2}(K)}
  +
  |w|^2_{\Gamma^+},
\end{equation}
 where
\begin{equation*}
  |w|^2_{\Gamma^+}
  :=
  \int_{\Gamma^+} (\vec b \cdot \vec n) w^2.
\end{equation*}
Then the finite element discretisation is to find $u_h \in K_h$ such that
\begin{equation}\label{eq:discrete_inequality}
  a_h(u_h, v_h - u_h)
  \geq
  l_h(v_h - u_h) \quad \forall v_h \in K_h,
\end{equation}
where
\begin{equation*}
  l_h(v_h)
  :=
  \int_{\W} fv_h + \sum_{K\in \T}
  \delta_K \int_K f \left(\vec b \cdot \grad v_h\right).
\end{equation*}

\section{Error analysis}\label{sec:linear_analysis}

In this section we prove the main result of this work: an error bound for
$\tripleNorm{u - u_h}$, where $u_h$ is a bound-preserving approximation to $u
\in \operatorname{H}_-(\W)$. The key idea is that we can combine the weak forms of the
variational problem \eqref{eq:pde_variational_form} and the discrete
\emph{inequality} \eqref{eq:discrete_inequality} and use the coercivity and
continuity of $a_h$ and the consistency of the SUPG formulation to derive a best
approximation result. Consistency is the subject of the next lemma. 

\begin{lemma}[Consistency of the SUPG method]\label{lem:consistency}
  Let $u$ be the solution of problem
  \eqref{eq:pde_variational_form}. Then for any $v_h \in \fes$,
  \begin{equation*}
    l_h(v_h) - a_h(u, v_h) = 0.
  \end{equation*}
\end{lemma}

\begin{proof}
  Note that
  \begin{equation}
    \begin{split}
      a_h(u,v_h)
      &=
      \int_{\W}\left(\vec b \cdot \grad u + c u\right)v_h 
      +
      \sum_{K\in \T}\delta_K
      \int_{K}\left(\vec b \cdot \grad u + c u\right)\vec b \cdot \grad v_h
      \\
      &=
      \int_{\W} f v_h 
      + 
      \sum_{K\in \T}\delta_K
      \int_{K} f \vec b \cdot \grad v_h
      =
      l_h (v_h),
    \end{split}
  \end{equation}
  by \eqref{eq:pde_variational_form}.
\end{proof}

\begin{lemma}[Continuity \& coercivity properties of $a_h$]\label{lem:cont_coerc}
  Assume that, for all $K \in \T$, the SUPG parameters $\delta_K$ are chosen so
  that
  \begin{equation}\label{eq:delta_k_upper}
    0 
    \leq
    \delta_K 
    \leq 
    \frac{\mu}{\Norm{c}^2_{\leb{\infty}(K)}}.
  \end{equation}  
  Then the bilinear form $a_h$ is coercive over $\operatorname{H}_-$
  with respect to the norm $\tripleNorm{\cdot}$, that is, for $w\in
  \operatorname{H}_-$
  \begin{equation}
    a_h(w, w) 
    \geq 
    \frac 1 2 \tripleNorm{w}^2.
  \end{equation}
  Furthermore, for $w\in \operatorname{H}_-$ we define the norm
  \begin{equation}\label{eq:star_norm}
    \tripleNorm{w}_*^2 
    := 
    \tripleNorm{w}^2 + \sum_{K \in \T}\delta_K^{-1}\Norm{w}^2_{\leb{2}(K)}.
  \end{equation}
  Then, for $w, v\in \operatorname{H}_-$ $a_h$ satisfies the
  continuity result
  \begin{equation}\label{eq:a_h_continuous}
    a_h(w, v) 
    \leq 
    C 
    \tripleNorm{w}_* 
    \tripleNorm{v}.
  \end{equation}
\end{lemma}

\begin{proof}
  We first show coercivity. Using (\ref{eq:green_formula}) yields
  \begin{equation}
    \begin{split}
    a_h(w, w) 
    \geq 
    \inf_{x \in \W}\left(c - \frac 1 2 \grad \cdot \vec b\right)
    \Norm{w}^2_{\leb{2}(\W)}
    +
    \frac 1 2 \int_{\partial \W}(\vec b \cdot \vec n)\lvert w \rvert^2
    +
    \sum_{K\in \T}
    \Norm{\delta_K^{\frac 1 2}\vec b \cdot \grad w}^2_{\leb{2}(K)}\\
    -
    \left\lvert
      \sum_{K\in \T}\delta_K\int_K c w (\vec b \cdot \grad w)
    \right\rvert,
    \end{split}
  \end{equation}
  where only the final term on the right hand side remains to be estimated. To
  this end, we observe that, on any $K \in \T$,
  \begin{equation}
    \left\lvert\int_K \delta_K c w (\vec b \cdot \grad w)\right\rvert
    \leq 
    \delta_K^{\frac 1 2} 
    \Norm{c}_{\leb{\infty}(K)}
    \Norm{w}_{\leb{2}(K)}
    \Norm{\delta_K^{\frac 1 2} \vec b \cdot \grad w}_{\leb{2}(K)}.
  \end{equation}
  We now split this product using Young's inequality and note the
  choice of $\delta_K$, given in \eqref{eq:delta_k_upper} to see that
  \begin{equation}
    \begin{split}
    \left\lvert\int_K \delta_K c w (\vec b \cdot \grad w)\right\rvert
    &\leq 
    \frac 1 2 \delta_K \Norm{c}^2_{\leb{\infty}(K)}\Norm{w}^2_{\leb{2}(K)}
    +
    \frac 1 2 
    \Norm{\delta_K^{\frac 1 2} \vec b \cdot \grad w}^2_{\leb{2}(K)}\\
    &\leq \frac 1 2 \mu \Norm{w}^2_{\leb{2}(K)}
    +
    \frac 1 2 \Norm{\delta_K^{\frac 1 2} \vec b \cdot \grad w}^2_{\leb{2}(K)},
    \end{split}
  \end{equation}
  from which the result follows. To show continuity, we again use Green's
  formula and bound 
  \begin{equation}
    \begin{split}
      a_h(w, v) 
      =& 
      \int_{\W}\left(\vec b \cdot \grad w + c w\right)v 
      +
      \sum_{K\in \T}\delta_K
      \int_{K}\left(\vec b \cdot \grad w + c w\right)
      \vec b \cdot \grad v
      \\
      &= \int_{\W}\left(c -  \operatorname{div}\vec b\right)w v
      -
      \int_{\W}
      \left(\vec b \cdot \grad v\right)
       w
      +
      \int_{\partial \W}(\vec b \cdot \vec n)w v \\
      &+ 
      \sum_{K\in \T}
      \int_K\left(\delta^{\frac 1 2}_K \vec b \cdot \grad w\right)
      \left( \delta^{\frac 1 2}_K \vec b \cdot \grad v\right)
      +
      \sum_{K\in \T}\delta_K\int_K c w
      \left(  \vec b \cdot \grad v\right).
    \end{split}
  \end{equation}
  An application of Cauchy-Schwarz results in
  \begin{equation}\label{eq:continuity_1}
    \begin{split}
      \lvert a_h(w, v) \rvert
      &\leq 
      \Norm{c - \operatorname{div} \vec b}_{\leb{\infty}(\W)}
      \Norm{w}_{\leb{2}(\W)}\Norm{v}_{\leb{2}(\W)}
      +
      \sum_{K\in \T}
      \Norm{\delta^{\frac 1 2}_K \vec b \cdot \grad v}_{\leb{2}(K)}
      \Norm{\delta^{-\frac 1 2}_K w}_{\leb{2}(K)}\\
      &\qquad + 
      \lvert w \rvert_{\Gamma_+}
      \lvert v \rvert_{\Gamma_+}
      +
      \sum_{K\in \T}
      \Norm{\delta^{\frac 1 2}_K \vec b \cdot \grad w}_{\leb{2}(K)}
      \Norm{\delta^{\frac 1 2}_K \vec b \cdot \grad v}_{\leb{2}(K)}\\
      &\qquad +
      \sum_{K\in \T}\Norm{c}_{\leb{\infty}(K)}
      \delta^{\frac 1 2}_K
      \Norm{ w}_{\leb{2}(K)}
      \Norm{\delta^{\frac 1 2}_K \vec b \cdot \grad v}_{\leb{2}(K)}.
    \end{split}
  \end{equation}
  Finally, let $M = \max\left\{1, \Norm{c - \operatorname{div} \vec
    b}_{\leb{\infty}(\W)},\Norm{c}_{\leb{\infty}(\W)}
  \max_K\delta_K^{1 \slash 2}\right\}$. Then applying a discrete
  Cauchy-Schwarz inequality to \eqref{eq:continuity_1} gives
  \begin{equation}
    \lvert a_h(w, v) \rvert
    \leq 
    2\frac M \mu
    \tripleNorm{w}_* 
    \tripleNorm{v},
  \end{equation}
  as required.
\end{proof}

\begin{Theorem}[Best approximation]\label{thm:best_approx}
  Let $u \in \operatorname{H}_-(\W)$ be the solution of
  \eqref{eq:pde_variational_form} and let $u_h \in K_h$ be the
  solution of \eqref{eq:discrete_inequality}. Then we have the
  following best approximation result:
  \begin{equation*}
    \tripleNorm{u-u_h}
    \leq
    C\inf_{v_h \in K_h}\tripleNorm{u-v_h}_*.
  \end{equation*}
\end{Theorem}

\begin{proof}
  We begin our estimation noting that, by coercivity of $a_h$ with
  respect to the norm $\tripleNorm{\cdot}$,
  \begin{equation}\label{eq:linear_coercivity}
    \frac 1 2 \tripleNorm{u-u_h}^2 
    \leq 
    a_h(u - u_h, u - u_h).
  \end{equation}
  By consistency, shown in Lemma \ref{lem:consistency}, we have
  \begin{equation}\label{eq:linear_continuous}
    a_h(u, u_h - v_h) 
    =
    l_h(u_h - v_h),
  \end{equation}
  and since $u_h$ satisfies the discrete problem \eqref{eq:discrete_inequality}
  we have 
  \begin{equation}\label{eq:linear_discrete}
    a_h(u_h, u_h - v_h) 
    \leq
    l_h(u_h - v_h).
  \end{equation}
  Equations \eqref{eq:linear_continuous} and
  \eqref{eq:linear_discrete} together imply that 
  \begin{equation*}
    a_h(u - u_h, u_h - v_h) 
    \geq 
    0.
  \end{equation*}
  The inequality \eqref{eq:linear_coercivity} therefore leads to
  \begin{equation*}
    \frac 1 2 \tripleNorm{u-u_h}^2 
    \leq 
    a_h(u - u_h, u - v_h) 
    \leq 
    2\frac M \mu
    \tripleNorm{u - v_h}_*\tripleNorm{v_h - u_h},
  \end{equation*}
  where we have used the continuity result \eqref{eq:a_h_continuous},
  concluding the proof.
\end{proof}

\begin{remark}[Relation to Falk's estimate {\cite{falk1974error}}]
  In the context of finite element approximations for variational
  inequalities, the equality stated in Lemma \ref{lem:consistency}
  does not generally hold. Instead, approximation results typically
  involve additional terms arising from the projection of the exact
  solution onto a cone, as discussed in \cite[Theorem
    2.1]{brezzi1977error}. However, in our setting, the exact solution
  is not projected onto a cone, whereas the discretisation is. This
  distinction eliminates certain terms that would otherwise appear on
  the right-hand side of the analogous error bound for elliptic
  variational inequalities.
\end{remark}

\begin{corollary}[Rates of convergence]\label{cor:rates_1} Let $\vec b \in
  \sob{1}{\infty}(\W)$, $c \in \leb{\infty}(\W)$, and let $u$ be the
  unique solution of \eqref{eq:pde_variational_form}, with $u_h \in
  K_h$ the solution of \eqref{eq:discrete_inequality}. Let $k \geq 1$,
  and assume that $u \in \sobh{r}(\W)$, where $r > \tfrac d2$ is sufficiently
  large so that $u$ is regular enough to belong to the domain of the
  Lagrange interpolation operator. Suppose that $C_{\delta}>0$ is
  sufficiently small so that
    \begin{equation*}
      \delta_K
      :=
      C_{\delta}h_K 
      \leq 
      \frac{\mu}{\Norm{c}^2_{\leb{\infty}(K)}}.
    \end{equation*}    
    Then, there exists a constant $C > 0$ independent of $h$ such that
    \begin{equation}
      \tripleNorm{u - u_h}
      \leq 
      C h^{\min\left\{k+1, r\right\}-\frac 1 2} \vert u \rvert_{\sobh{r}(\W)}.
    \end{equation}
  \end{corollary}

  \begin{proof}
    Since $u(x) \in [0,1]$ for all $x$, the piecewise polynomial
    Lagrange interpolant of $u$, $\mathcal{I}u$ onto the finite
    element space associated with any sufficiently regular
    triangulation of $\W$ also has $\mathcal{I}u(x_i) \in [0,1],
    i=1,...,N$. We therefore have $\mathcal{I}u \in K_h$ and therefore
    we can take $v_h = \mathcal{I}u$ in Theorem \ref{thm:best_approx}
    and apply standard interpolation estimates \cite{EG21-I} to the
    terms in the $\tripleNorm{\cdot}_*$ norm. With $C_{\mathcal{I}}$
    and $C_{\text{tr}}$ the constants in the interpolation and trace
    estimates respectively, one sees that
    \begin{equation}\label{eq:interp_first}
      \Norm{u-\mathcal{I}u}_{\leb{2}(\W)}
      \leq 
      C_{\mathcal{I}} h^{r} \lvert u \rvert_{\sobh{r}(\W)},
    \end{equation}
    \begin{equation}
      \Norm{\delta_K^{\frac 1 2}\vec b \cdot \grad 
      \left(u-\mathcal{I}u\right)}_{\leb{2}(K)}
      \leq 
      C_{\mathcal{I}} \Norm{\vec b}_{\leb{\infty}(K)} C_{\delta}^{\frac 1 2} 
      h^{r-\frac 1 2} \lvert u \rvert_{\sobh{r}(\W)},
    \end{equation}

    \begin{equation}
      \delta_K^{-\frac 1 2}\Norm{u - \mathcal{I}u}_{\leb{2}(K)}
      \leq 
      C_{\mathcal{I}} C_{\delta}^{-\frac 1 2} h^{r-\frac 1 2 }
       \lvert u \rvert_{\sobh{r}(\W)}.
    \end{equation}

    For the boundary norm, we appeal to a trace estimate, followed by estimates
    for the Lagrange interpolation error:
    \begin{equation}\label{eq:interp_last}
      \begin{split}
      \lvert u - \mathcal{I}u \rvert_{\Gamma^+\cap K}
      &\leq 
      C_{\text{tr}}\left(h_K^{-\frac 1 2}\Norm{u - \mathcal{I}u}_{\leb{2}(K)}
      + 
      h_K^{\frac 1 2}\Norm{\grad \left(u - \mathcal{I}u\right)}_{\leb{2}(K)}
      \right)\\
      &\leq 
      2 C_{\text{tr}} C_{\mathcal{I}} 
      h_K^{r-\frac 1 2}
      \lvert u \rvert_{\sobh{r}(\W)}.
      \end{split}
    \end{equation}
    After collecting Equations \eqref{eq:interp_first}-\eqref{eq:interp_last},
    the result follows.
  \end{proof}

\section{Extension to the case of nonlinear reaction}\label{sec:nonlinear}

We now consider the extension of the analysis of the previous sections
to the problem
\begin{equation}\label{eq:nonlinear_pde}
  \begin{split}
  \vec b \cdot \grad u + |u|^{p-2}u &= f \quad \text{in}\,\, \W\\
  u &= 0 \quad \text{on}\,\, \Gamma_-,
  \end{split}
\end{equation}
where $1 < p \leq 2$. The problem therefore has nonlinear reaction
which is singular as $u \to 0$, meaning that preservation of
positivity of solutions becomes very important.

Define (cf. Equation \eqref{eq:graph_space})
\begin{equation*}
  \operatorname{H}_{-,p}(\W)
  = 
  \{v \in \leb{2}(\W) : \vec b \cdot \grad v + |v|^{p-2}v \in \leb{2}(\W), \,\,
  (\vec b \cdot \vec n) v = 0 \,\,\text{on}\,\, \Gamma_-\}.
\end{equation*}
Then the weak form of problem \eqref{eq:nonlinear_pde} is to find $u \in
\operatorname{H}_{-,p}(\W)$ such that 
\begin{equation}\label{eq:nonlinear_weak_form}
  a(u, v)
  + 
  b(u;u,v)
  =
  l(v) \quad \forall v \in \leb{2}(\W),
\end{equation}
where $a(\cdot,\cdot)$ and $l(\cdot)$ are as defined in \eqref{eq:a_bilinear} and
\eqref{eq:l_linear} (the former with $c \equiv 0$), and where the semilinear form $b$ is
defined to be 
\begin{equation}\label{eq:semilinear_form}
  b(u;v,w)
  =
  \int_{\W}\lvert u \rvert^{p-2}vw \diff x.
\end{equation}

\begin{remark}[Existence and uniqueness]
The existence and uniqueness of solutions to this problem, in the
presence of diffusion, can be established using monotone operator
theory, as detailed in \cite[Chapter
  10]{renardy2006introduction}. While the proof techniques for the
corresponding linear problem do not straightforwardly generalise to
the nonlinear case, they can be adapted through fixed-point arguments.
\end{remark}

Classical error analysis in the $\leb{p}(\W)$ norm often leads to
suboptimal estimates for $p\neq 2$. To address this issue, quasi-norms
were introduced in \cite{barrett1993finite} as a tool for achieving
optimal convergence rates. We define the following quasi-norm, for
fixed $w \in \leb{p}(\W)$, we let
\begin{equation}
  \Norm{v}^2_{(w,p)}
  := 
  \int_{\W}
  \lvert v\rvert^2\left(\lvert v \rvert + \lvert w \rvert\right)^{p-2} \diff x,
\end{equation}
for all $v \in \leb{p}(\W)$. 

\subsection{Preliminary results}
Error analysis in quasi-norms requires technical lemmata which are stated here for
convenience. The key results are strong monotonicity and boundedness properties of the
form $b$ with respect to the quasi-norm, given in Lemma
\ref{lem:coercivity} and Lemma \ref{lem:quasi_boundedness} respectively.

\begin{lemma}(\cite[Lemma 2.1]{barrett1993finite}).
  Let $p >1$ and $x,y \in \reals$. Then there exist positive constants $C_1,
  C_2$ depending only upon $p$ such that 
  \begin{equation}\label{eq:technical_inequality_1}
    \left\lvert |x|^{p-2} x - \lvert y\rvert^{p-2} y\right\rvert
    \leq
    C_1
    \left(\left\lvert x\right\rvert + \left\lvert y \right\rvert\right)^{p-2}
    \left\lvert x - y\right\rvert,
  \end{equation}
  \begin{equation}
    \left(|x|^{p-2} x - \lvert y\rvert^{p-2} y\right)\left(x - y\right)
    \geq 
    C_2 
    \left(\lvert x\rvert + \lvert y \rvert\right)^{p-2}
    \left\lvert x - y\right\rvert^2,
  \end{equation}
\end{lemma}

\begin{lemma}\label{lem:coercivity}
  There exists $C_C >0$ such that, for all $u,v\in \leb{p}(\W)$,
  \begin{equation}\label{eq:coercivity}
    b(u;u, u-v) - b(v; v, u-v)
    \geq 
    C_C
    \Norm{u-v}^2_{(u,p)}.
  \end{equation}
\end{lemma}

\begin{proof}
  We first note that 
  \begin{equation}
    \lvert x \rvert + \lvert y \rvert
    =
    \lvert x \rvert + \lvert y - x + x \rvert
    \leq 
    2 \lvert x \rvert + \lvert x-y \rvert,
  \end{equation}
  We therefore we have, noting that $p < 2$,
  \begin{equation}
    \begin{split}
    \left(|x|^{p-2} x - \lvert y\rvert^{p-2} y\right)\left(x - y\right)
    &\geq
    C_2\left(\lvert x \rvert + \lvert y \rvert\right)^{p-2}
    \left\lvert x - y\right\rvert^2\\
    &\geq 
    2^{p-2}C_2 \left(\lvert x \rvert + \lvert x-y \rvert\right)^{p-2}
    \left\lvert x - y\right\rvert^2,
    \end{split}
  \end{equation}
  and so the result follows with $C_C = 2^{p-2}C_2$.
\end{proof}

\begin{lemma}\label{lem:quasi_boundedness} For $1 < p < 2$, there exists $C_B >0$ such
  that for any $\theta \in (0, 1]$ and any $u,v,w \in \leb{p}(\W)$,
  \begin{equation}
    \lvert b(u;u, w) - b(v; v, w) \rvert
    \leq 
    C_B
    \left(
    \theta \Norm{u-v}^2_{(u,p)}
    +
    \theta^{-1} \Norm{w}^2_{(u,p)}
    \right).
  \end{equation}
\end{lemma}

\begin{proof}
  For real numbers $x, y$ and $z$, we apply \eqref{eq:technical_inequality_1} to
  see that 
  \begin{equation}
    \begin{split}
    \left\lvert |x|^{p-2} x - \lvert y\rvert^{p-2} y\right\rvert |w|
    &\leq
    C_1
    \left(\left\lvert x\right\rvert + \left\lvert y \right\rvert\right)^{p-2}
    \left\lvert x - y\right\rvert|w|\\
    &\leq
    C'_1
    \left(\left\lvert x\right\rvert + \left\lvert x-y \right\rvert\right)^{p-2}
    \left\lvert x - y\right\rvert|w|,
    \end{split}
  \end{equation}
  which can be further bounded by invoking \cite[Lemma 2.2]{liu2000finite} with
  $a = |x|, \sigma_1=|x-y|, \sigma_2=|z|,\alpha=1$ and $r = p$. This gives the
  bound
  \begin{equation}
    \begin{split}
    \left\lvert |x|^{p-2} x - \lvert y\rvert^{p-2} y\right\rvert \lvert z\rvert 
    &\leq
    C_B \left(
      \theta \left(|x| + |x-y|\right)^{p-2}|x-y|^2
      +
      \theta^{-1}\left(|x| + |z|\right)^{p-2}|z|^2\right),
    \end{split}
  \end{equation}
  from which the desired result is immediate.
\end{proof}

\subsection{Finite element discretisation}

We present the finite element method for the nonlinear problem as an extension of the
linear case by treating the advection with SUPG-stabilisation. 
The discrete problem is then to find $u_h \in
K_h$ such that 
\begin{equation}\label{eq:nonlinear_fem}
  a_h(u_h, v_h - u_h) + b(u_h; u_h, v_h - u_h) 
  \geq 
  l_h(v_h - u_h) \quad \forall v_h \in K_h,
\end{equation}
where $a_h$ is as defined in Equation \eqref{eq:supg_bilinear_form} with $c
\equiv 0$ and $b$ is the semilinear form defined in Equation
\eqref{eq:semilinear_form}.

\begin{remark}[Weaker control of the advective term]\label{rem:no_l2}
  In the following analysis, we prove a best approximation result by
  treating the advective part in an analogous way to Theorem
  \ref{thm:best_approx}, and the nonlinear part separately with the
  quasi-norm $\Norm{\cdot}_{(u,p)}$. Since we are then working with
  the SUPG bilinear form $a_h$ with the reaction coefficient $c$ set
  to be identically zero, we can no longer assume the coercivity
  condition \eqref{eq:coercivity_assumption}. As a result, the
  coercivity shown in Lemma \ref{lem:cont_coerc} is replaced with the
  weaker notion
  \begin{equation}\label{eq:weaker_norm}
    a_h(v_h, v_h) 
    =
    \lvert v_h\rvert^2_{\Gamma_+} 
    + 
    \sum_{K \in \T}\Norm{\delta_K^{\frac 1 2}\vec b \cdot \grad v_h}^2_{\leb{2}(K)}.
  \end{equation}
  We therefore note that control of the error in $\leb{2}(\W)$ is
  lost, or rather replaced, with error control in
  $\Norm{\cdot}_{(u,p)}$.

  The weak norm given on the right hand side of Equation
  \eqref{eq:weaker_norm} is the norm $\tripleNorm{\cdot}$
  (cf. Equation \eqref{eq:triplenorm}) with $\mu$ set to zero.  For
  the remainder of this section we will reuse $\tripleNorm{\cdot}$ to
  denote this weaker norm, and adopt a similar convention for
  $\tripleNorm{\cdot}_*$.
\end{remark}

\begin{lemma}[Quantification of inconsistency (cf. Lemma \ref{lem:consistency})]
  \label{lem:inconsistency}
  Let $u$ be the solution to problem \ref{eq:nonlinear_weak_form}. Then for any $w_h \in
  \fes$,
  \begin{equation}\label{eq:consistency_nonlinear}
    l_h(w_h) 
    -
    a_h(u, w_h) 
    -
    b(u;u,w_h)
    =
    \sum_{K \in \T} \delta_K 
    \int_K|u|^{p-2}u\left(\vec b \cdot \grad w_h\right).
  \end{equation}
\end{lemma}

\begin{proof}
  The result follows after evaluating the left hand side of
  \autoref{eq:consistency_nonlinear} to obtain
  \begin{equation}
    \int_{\W}\left(f - \vec b \cdot \grad u - \lvert u \rvert^{p-2}u\right)w_h 
    + 
    \sum_{K \in \T}\delta_K \int_K \left(f - \vec b \cdot \grad u \right)
    \left(\vec b \cdot \grad w_h\right),
  \end{equation}
  and using the fact that $u$ is the solution of \autoref{eq:nonlinear_pde}.
\end{proof}

\begin{Theorem}[Best approximation]
  Assume that $\operatorname{div}\vec b = 0$, that $u$ is the solution of
  \eqref{eq:nonlinear_weak_form}, and that $u_h$ is the finite element solution, that
  is, the solution of \eqref{eq:nonlinear_fem}. Then
  \begin{multline}
    \tripleNorm{u-u_h}^2
    +
    \Norm{u-u_h}^2_{(u,p)}
    \leq
    C \inf_{v_h \in K_h}
    \left(\Norm{u-v_h}^2_{(u,p)} + \tripleNorm{u-v_h}^2_*\right)
    \\
    +
    \sup_{0\neq w_h\in \fes}
    \frac{\sum_{K \in \T} \delta_K 
      \int_K|u|^{p-2}u\left(\vec b \cdot \grad w_h\right)}
         {\tripleNorm{w_h}}.
  \end{multline}      
\end{Theorem}

\begin{proof}
  From the coercivity result
  \eqref{eq:coercivity}, and in light of Equation \eqref{eq:weaker_norm}, we immediately
  have
  \begin{equation}\label{eq:nonlin_first}
    \begin{split}
    \tripleNorm{u - u_h}^2
    +
    C_C \Norm{u-u_h}^2_{(u,p)}
    \leq 
    a_h(u-u_h,u-u_h)
    +
    b(u;u, u-u_h) - b(u_h;u_h, u-u_h).
    \end{split}
  \end{equation}
  We now wish to introduce an arbitrary $v_h$. To this end, invoking
  Lemma \ref{lem:inconsistency} with $w_h = u_h - v_h$, and using the
  fact that $u_h$ satisfies Equation \eqref{eq:nonlinear_fem}, we
  conclude that
  \begin{equation}
    \begin{split}
    a_h(u-u_h, u_h - v_h) 
    + 
    b(u; \,&u, u_h - v_h) 
    - 
    b(u_h; u_h, u_h - v_h)\\
    &+
    \sum_{K\in\T}\delta_K \int_K \lvert u \rvert^{p-2}u 
    \left(\vec b \cdot \grad (u_h - v_h)\right) \geq 0.
    \end{split}
  \end{equation}  
  Then (cf. Theorem \ref{thm:best_approx}) we can write 
  \begin{equation}
    \begin{split}
    a_h(u - u_h,u-u_h) + b(u;u,& \,u-u_h)- b(u_h;u_h,u-u_h)\\
    \leq 
    a_h(u - u_h, \,&u-v_h) + b(u;u,u-v_h)- b(u_h;u_h,u-v_h)\\
    +
    &\sum_{K\in\T}\delta_K \int_K \lvert u \rvert^{p-2}u 
    \left(\vec b \cdot \grad (u_h - v_h)\right).
    \end{split}
  \end{equation}
  Using Lemma \ref{lem:quasi_boundedness}
  \begin{equation}
    b(u;u,u-v_h) - b(u_h;u_h,u-v_h)
    \leq 
    C_B\left(
    \theta \Norm{u-u_h}^2_{(u,p)}
    +
    \theta^{-1} \Norm{u-v_h}^2_{(u,p)}
    \right).
  \end{equation}
  We apply a boundedness argument analogous to Lemma
  \ref{lem:cont_coerc} to see
  \begin{equation}\label{eq:nonlin_last}
    a_h(u - u_h, u-v_h) 
    \leq 
    \tripleNorm{u-v_h}_* \tripleNorm{u-u_h}
    \leq 
    \frac 1 2\tripleNorm{u-v_h}_*^2
    +
    \frac 1 2\tripleNorm{u-u_h}^2.
  \end{equation}
  Upon making the choice of $\theta = \min\left(\frac{C_C}{2 C_B}, \frac 1 2 \right)$,
  collecting Equations \eqref{eq:nonlin_first}-\eqref{eq:nonlin_last}, we have
  \begin{equation}
    \tripleNorm{u-u_h}^2
    +
    \Norm{u-u_h}^2_{(u,p)}
    \leq
    \tripleNorm{u-v_h}^2_*
    + 
    \Theta \Norm{u-v_h}^2_{(u,p)}
    +
    \sum_{K\in\T}\delta_K \int_K \lvert u \rvert^{p-2}u 
    \left(\vec b \cdot \grad (u_h - v_h)\right),
  \end{equation}
  where $\Theta = \max \left(2 \frac{C_B}{C_C},2
  \frac{C^2_B}{C^2_C}\right)$, concluding the proof.
\end{proof}

\section{Overview of active-set techniques and applications}
\label{sec:iterative_methods}

The variational inequality \eqref{eq:discrete_inequality}, posed over
the finite-dimensional space $K_h$, is nonlinear and must be solved
using an iterative method. In this section, we describe solution
techniques for this class of problems. Throughout, we denote a vector
$\vec{\xi} \in \reals^N$ as non-negative, written $\vec{\xi} \geq 0$,
if all its components satisfy $(\vec{\xi})_i \geq 0$.

The problem \eqref{eq:discrete_inequality} can be viewed as a specific
instance of the general variational inequality
\begin{equation} 
  \ltwop{G(u)}{v - u} \geq 0 \quad \forall v \in K,
\end{equation}
where $G: V \to V$ is a mapping on a Banach space $V$ and $K$ is a
closed and convex subset of $V$. In the discrete setting, we introduce
a discrete operator $\bar{G}_h : \fes \to \fes$ and reformulate
\eqref{eq:discrete_inequality} as: find $u_h \in K_h$ such that:
\begin{equation} 
  \ltwop{\bar{G}_h(u_h)}{v_h - u_h} \geq 0 \quad \forall v_h \in K_h.
\end{equation}
An equivalent reformulation involves a mapping $G_h : \reals^N \to
\reals^N$, where $N$ denotes the dimension of the finite element
space. This mapping is defined using the finite element system matrix
and right-hand side. Specifically, let $\vec{A}_h$ represent the
assembled system matrix such that
\begin{equation}
  \left(\vec{A}_h\right)_{ij} = a_h(\varphi_j, \varphi_i),
\end{equation}
with $\{\varphi_i\}_{i=1}^N$ the nodal basis functions. Similarly, let
\begin{equation}
  \left(\vec{F}_h\right)_i = l_h(\varphi_i).
\end{equation}
If $\vec{V} \in \reals^N$ is the vector of degrees of freedom
corresponding to a finite element function $v_h$, given by:
\begin{equation}
  v_h(\vec{x}) = \sum_{i=1}^N \vec{V}_i \varphi_i(\vec{x}),
\end{equation}
then the mapping $G_h : \mathbb{R}^N \to \mathbb{R}^N$ is given by:
\begin{equation}
  G_h(\vec{V}) = \vec{A}_h \vec{V} - \vec{F}_h.
\end{equation}
We also define the discrete set
\begin{equation}
  K_{h, N} := \{\vec{V} \in \mathbb{R}^N : 0 \leq \vec{V}_i \leq 1, \, i = 1, \dots, N\}.
\end{equation}
The discrete variational inequality \eqref{eq:discrete_inequality} can
now be reformulated algebraically as: Find \(\vec{U} \in K_{h, N}\)
such that
\begin{equation}
  G_h(\vec{U}) \cdot (\vec{V} - \vec{U}) \geq 0 \quad \forall \vec{V} \in K_{h, N}.
\end{equation}

\begin{definition}[Projection Operators]
Let $P : \fes \to K_h$ denote the orthogonal projection onto the
closed convex subset $K_h$. Assuming $K_h$ is a box as defined in
\eqref{eq:definition_of_K_h}, the projection $P$ has a simple
explicit form. For $v_h \in K_h$, we define:
\begin{equation}\label{eq:proj}
  P(v_h) := \sum_{i=1}^N \max\left\{0, \min\left\{v_h(\vec{x}_i), 1\right\}\right\} \varphi_i.
\end{equation} 
Similarly, we define the projection in the Euclidean setting, $\Pi :
\mathbb{R}^N \to K_{h, N}$, as:
\begin{equation}\label{eq:proj_Rn}
  \Pi(\vec{X})_i := \max\left\{0, \min\left\{\vec{X}_i, 1\right\}\right\}.
\end{equation}
\end{definition}

\subsection{Projection Methods}

For discretisations of many elliptic variational inequalities, a
common approach combines iterative schemes, such as Richardson
iteration or successive over-relaxation (SOR), with a projection
step. These methods ensure that the iterates remain within the convex
set and are relatively straightforward to implement. The underlying
principle for such methods is that the discrete problem can often be
reformulated as a minimisation problem, allowing the application of
various well-established optimisation algorithms.

However, when the finite element system matrix is not symmetric, as is
the case in this problem, these classical methods generally fail to
converge. This limitation necessitates the use of projection methods
that are robust under broader assumptions. One such method, introduced
in \cite{alfredo1994iterative}, iteratively projects the solution back
onto the feasible set and is defined as follows. Given an initial
guess $u_h^0 \in K_h$, the method updates the iterate using:
\begin{equation}
  u_h^{j+1} = P\left(u_h^j - \gamma \bar{G}_h(u_h^j)\right),
\end{equation}
where \(\gamma > 0\) is a step size parameter. Iteration is terminated
when the stopping criterion \(\Norm{u_h^{j+1} - u_h^j} <
\texttt{TOL}\) is satisfied, for a given tolerance \(\texttt{TOL}\).

The projection method, as detailed in \cite{alfredo1994iterative},
converges to a solution under the assumptions that the operator is
monotone and Lipschitz continuous, without requiring
differentiability. While this method is simple and guarantees
convergence, it often requires many iterations to achieve acceptable
accuracy. Consequently, for large-scale problems or those demanding
high efficiency, more advanced and scalable methods are necessary to
enhance convergence rates and reduce computational costs.

\subsection{Reduced-Space Active Set Method}

The numerical results presented in this work were obtained using the
finite element discretisation package Firedrake
\cite{FiredrakeUserManual}, which leverages PETSc
\cite{balay1998petsc} for solving the variational inequalities
introduced in Sections~\ref{sec:problem_setup} and
\ref{sec:nonlinear}. Among the various solvers available in PETSc, we
employed the reduced-space active-set method
\cite{benson2006flexible}. This method is particularly well-suited for
parallel implementations and guarantees that the computed solution
satisfies the imposed constraints. Although no formal convergence
proofs were provided in \cite{benson2006flexible}, the method has
demonstrated efficiency and robustness in numerous test cases,
especially in applications involving monotone operators, which aligns
with the scope of this work.

Let $\vec{X} \in \reals^N$. Define the active and inactive sets
associated with $\vec{X}$ as follows:
\begin{align}
  A_0(\vec{X}) &:= \left\{ i : \vec{X}_i = 0 \text{ and } G_h(\vec{X})_i > 0 \right\}, \\
  A_1(\vec{X}) &:= \left\{ i : \vec{X}_i = 1 \text{ and } G_h(\vec{X})_i \leq 0 \right\}, \\
  A(\vec{X}) &:= A_0(\vec{X}) \cup A_1(\vec{X}), \\
  I(\vec{X}) &:= \{1, 2, \dots, N\} \setminus A(\vec{X}),
\end{align}
where $G_h(\vec{X})_i$ denotes the $i^{\text{th}}$ component of the
nonlinear residual evaluated at $\vec{X}$. The set $A(\vec{X})$
contains the indices of degrees of freedom where the box constraints
are active, while $I(\vec{X})$ denotes the inactive set.

In the reduced-space active set method, starting from an initial guess
$\vec{X}^0$, the degrees of freedom corresponding to $A(\vec{X}^j)$
are fixed at each iteration. A linear system, derived from the
linearisation of $G_h$, is solved over the inactive set $I(\vec{X}^j)$
to compute an update $\delta \vec{X}^j$, analogous to Newton's
method. Specifically, let $G_h^I(\vec{X}^j)$ denote the reduced
residual vector containing only the entries indexed by $I(\vec{X}^j)$,
and let $J^I$ be the Jacobian of $G_h^I$ restricted to the inactive
set. The update is then decomposed as
\begin{equation}
  \delta \vec{X}^j = \delta^I \vec{X}^j \oplus \delta^A \vec{X}^j,
\end{equation}
where $\delta^A \vec{X}^j = \vec{0}$, and $\delta^I \vec{X}^j$ satisfies
\begin{equation}\label{eq:solve_step}
  J^I \delta^I \vec{X}^j = -G_h^I(\vec{X}^j).
\end{equation}
The updated solution is then computed as
\begin{equation}\label{eq:reduced_space_update}
  \vec{X}^{j+1} = \Pi\left(\vec{X}^j + \alpha \delta \vec{X}^j\right),
\end{equation}
where $\Pi$ denotes the projection onto the constrained space
$K_{h,\,N}$, as defined in \eqref{eq:proj_Rn}. A line search is used
to determine the step size $\alpha$, employing either backtracking
or a secant-based algorithm provided by PETSc. For further details on
the selection of $\alpha$, we refer the reader to
\cite{benson2006flexible, zhu2021bound}.

Importantly, due to the projection in \eqref{eq:reduced_space_update},
all iterates satisfy the box constraints, i.e., $\vec{X}^j \in
K_{h,\,N}$ for all $j$. The algorithm terminates when either the
residual or the relative reduction in the residual falls below a
specified tolerance, set to $10^{-8}$ in this work.

\section{Numerical experiments}\label{sec:numerics}

In this section, we present a summary of numerical experiments that
validate the convergence rates predicted by the theoretical results of
the previous sections. Additionally, we illustrate the effectiveness
of the bound-preserving method in handling less smooth cases, where
rigorous theoretical guarantees are unavailable.

\subsection{Example 1: convergence to a smooth solution}\label{sec:example_1}

We first examine a test case for which we expect the exact solution to possess higher
regularity than functions in $\operatorname{H}_-(\W)$ (cf. Remark
\ref{rem:regularity_remarks}). To this end, let $\W = (0,1)\times (0,1)$, and let $\vec
b_1 = (1, \sqrt{2})$, so that the inflow boundary is
\begin{equation*}
  \Gamma^+ 
  = 
  \{(x, y) \in \partial \W : x = 0\} 
  \cup 
  \{(x, y) \in \partial \W : y =0\}.
\end{equation*}
Smooth inflow boundary data is prescribed:
\begin{equation}
  g_1(x, y) := \begin{cases}
    \exp\left(1 - \frac{1}{1 - 5(x-\frac 1 2)^2}\right) 
    & \text{if $|x-\frac 1 2| < \frac{1}{\sqrt{5}}$},\\
          0 & \text{otherwise}.
   \end{cases}
\end{equation}
Finally, let $c \equiv 1$ and let the right hand side $f$ be identically zero. Then an
exact solution $u_1 \in C^{\infty}(\overline{\W})$ can be found using the method of
characteristics. One finds that 

\begin{equation}\label{eq:u_1}
  u_1(x, y)
  =
  \begin{cases}
    g_1\left(x - \frac{y}{\sqrt{2}}, y\right)
    \exp\left( - \frac{y}{\sqrt{2}}\right) 
    & \text{if  $-\sqrt{\frac 2 5} 
       < y - x \sqrt{2} + \frac{1}{\sqrt 2}
       < \sqrt{\frac 2 5} $},\\
          0 & \text{otherwise}.
   \end{cases}
\end{equation} 
Convergence results of the bound-preserving finite element method for piecewise linear
and quadratic elements are shown in \autoref{fig:firedrake_rates_1}. In both cases the
theoretical convergence rate of $k + \tfrac 1 2$ shown in Corollary \ref{cor:rates_1} is
attained. The number of iterations required by the variational inequality solver is
shown in \autoref{fig:firedrake_iterations_1}. Note that with decreasing $h$ the number
required is either stable or decreasing.

\begin{figure}
  \begin{subfigure}{.48\linewidth}
\begin{tikzpicture}[scale=0.93]
  \begin{axis}[
      width =\linewidth,
      xmode=log, ymode=log,
      xmin=5e-4, xmax=1e-2,
      ymin=1e-6, ymax=5e-1,
      grid=both,
      major grid style={black!50},
      xlabel = \(h\),
      ylabel = \(\tripleNorm{u-u_h}\),
      legend style={at={(0.0,1)},anchor=north west}
  ]
  \addplot[solid, mark=square*, mark options={scale=1, solid}, color={black!100},
           line width=1.0] coordinates {
    (7.8125e-3, 1.18e-02)
    (3.91e-3, 4.36e-03)
    (1.95e-3, 1.56e-03)
    (9.77e-4, 5.51e-04)};
  \addlegendentry{\(k = 1\)}

  \addplot[solid, mark=*, mark options={scale=1,solid}, color={black!100}, 
           line width=1.0] coordinates {
    (7.8125e-3, 5.92e-04 )
    (3.91e-3, 1.10e-04)
    (1.95e-3, 1.98e-05)
    (9.77e-4, 3.51e-06)};
    \addlegendentry{\(k = 2\)}

    \addplot[dashed, mark = square*,
             mark options={scale=1,solid, fill = white, fill opacity=0},
             color={black!100}, line width=1.0] coordinates {
      (7.8125e-3, 8 * 6.91e-4)
      (3.91e-3, 8 * 2.44e-4)
      (1.95e-3, 8 * 8.61e-5)
      (9.77e-4, 8 * 3.05e-5)};
      \addlegendentry{\(\mathcal{O}(h^{3 \slash 2})\)}
  
    \addplot[dashed, mark =*,
             mark options={scale=1,solid,fill = white, fill opacity=0},
             color={black!100}, line width=1.0] coordinates {
      (7.8125e-3, 50 * 5.39e-06)
      (3.91e-3, 50 * 9.56e-07)
      (1.95e-3, 50 * 1.68e-07)
      (9.77e-4, 50 * 2.98e-08)};
      \addlegendentry{\(\mathcal{O}(h^{5/2})\)}

  \end{axis}
  \end{tikzpicture}
  \caption{\label{fig:firedrake_rates_1} Approximation error for Example 1,
  \S\ref{sec:example_1}, smooth solution.}
\end{subfigure}
\hfill
\begin{subfigure}{.48\linewidth}
  \begin{tikzpicture}[scale=0.93]
    \begin{axis}[
        width = \linewidth,
        xmode=log, ymode=log,
        xmin=5e-4, xmax=1e-2,
        ymin=1e-2, ymax=0.6,
        grid=both,
        major grid style={black!50},
        xlabel = \(h\),
        ylabel = \(\tripleNorm{u-u_h}\),
        legend style={at={(0.0,1)},anchor=north west}
    ]
    \addplot[solid, mark=square*, mark options={scale=1, solid}, color={black!100}, 
             line width=1.0] coordinates {
      (7.8125e-3, 1.94e-01)
      (3.91e-3, 1.45e-01)
      (1.95e-3, 1.09e-01)
      (9.77e-4, 8.09e-02)};
    \addlegendentry{\(k = 1\)}

    \addplot[solid, mark=*, mark options={scale=1,solid}, color={black!100}, 
             line width=1.0] coordinates {
      (7.8125e-3, 1.62e-01 )
      (3.91e-3, 1.18e-01)
      (1.95e-3, 8.5e-02)
      (9.77e-4, 6.13e-02)};
      \addlegendentry{\(k = 2\)}
    
  \addplot[dashed, mark = square*,mark options={scale=1,solid, fill = white, 
           fill opacity=0},color={black!100}, line width=1.0] coordinates {
    (7.8125e-3, 1.5 * 8.84e-2)
    (3.91e-3, 1.5 * 6.25e-2)
    (1.95e-3, 1.5 * 4.42e-2)
    (9.77e-4, 1.5 * 3.13e-2)};
    \addlegendentry{\(\mathcal{O}(h^{1 \slash 2})\)}

    \end{axis}
    \end{tikzpicture}
    \caption{\label{fig:firedrake_rates_2}Approximation error for Example 2,
    \S\ref{sec:example_2}, solution in \(\operatorname{H}_-(\W) \backslash
    \sobh{1}(\W)\) only.}
  \end{subfigure}
  \caption{\label{fig:linear_example_rates}Approximation errors for the bound-preserving
  finite element method in the full SUPG
  norm $\tripleNorm{\cdot}$. Polynomial degrees $k =1,2$ shown, yielding expected
  convergence rates of $3\slash 2$ and $5 \slash 2$ respectively in the smooth case. As
  expected when the solution has minimal regularity, theoretical rates are not attained
  and there is no benefit to increasing polynomial degree.}
\end{figure}
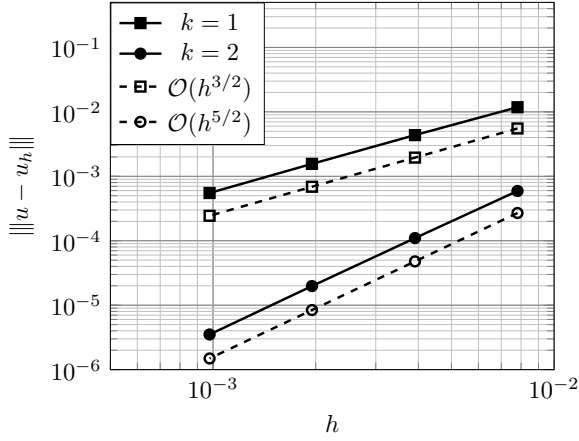
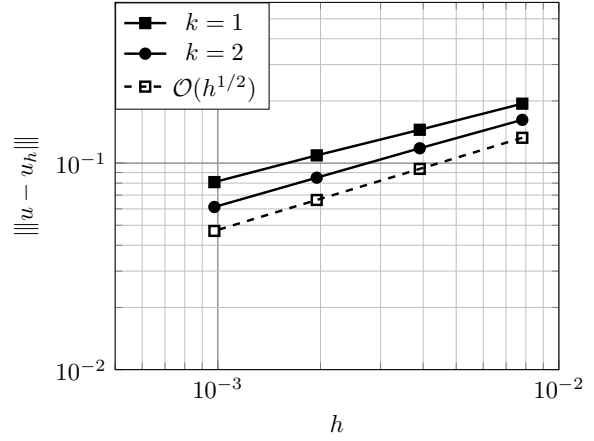

\begin{figure}
  \begin{subfigure}{.48\linewidth}
\begin{tikzpicture}
  \begin{axis}[
      width =\linewidth,
      xmode=log,
      xmin=5e-4, xmax=1e-2,
      ymin=0, ymax=25,
      grid=none,
      major grid style={black!50},
      xlabel = \(h\),
      ylabel = \(N\),
      legend style={at={(0.0,1)},anchor=north west}
  ]
  \addplot[solid, mark=square*, mark options={scale=1, solid}, color={black!100}, 
           line width=1.0] coordinates {
    (7.8125e-3, 7)
    (3.91e-3, 8)
    (1.95e-3, 9)
    (9.77e-4, 9)};
  \addlegendentry{\(k = 1\)}

  \addplot[solid, mark=*, mark options={scale=1, solid}, color={black!100}, 
           line width=1.0] coordinates {
    (7.8125e-3, 22)
    (3.91e-3, 22)
    (1.95e-3, 3)
    (9.77e-4, 1)};
  \addlegendentry{\(k = 2\)}
  \end{axis}
  \end{tikzpicture}
  \caption{\label{fig:firedrake_iterations_1} $u \in \sobh{3}(\W)$.}
\end{subfigure}
\hfill
\begin{subfigure}{.48\linewidth}
  \begin{tikzpicture}
    \begin{axis}[
        width = \linewidth,
        xmode=log, 
        xmin=5e-4, xmax=1e-2,
        ymin=0, ymax=60,
        grid=none,
        major grid style={black!50},
        xlabel = \(h\),
        ylabel = \(N\),
        legend style={at={(1,1)},anchor=north east}
    ]
    \addplot[solid, mark=square*, mark options={scale=1, solid}, color={black!100}, 
             line width=1.0] coordinates {
      (7.8125e-3, 10)
      (3.91e-3, 12)
      (1.95e-3, 15)
      (9.77e-4, 15)};
    \addlegendentry{\(k = 1\)}

    \addplot[solid, mark=*, mark options={scale=1,solid}, color={black!100}, 
             line width=1.0] coordinates {
      (7.8125e-3, 23 )
      (3.91e-3, 35)
      (1.95e-3, 47)
      (9.77e-4, 57)};
      \addlegendentry{\(k = 2\)}

    \end{axis}
    \end{tikzpicture}
    \caption{\label{fig:firedrake_iterations_2} \(u \in \operatorname{H}_-(\W)
    \backslash \sobh{1}(\W)\).}
  \end{subfigure}
  \caption{\label{fig:linear_example_iterations}Active-set reduced-space iteration
  counts until the residuals or their relative reduction reach $10^{-8}$. Note that for
  the smooth solution the number of iterations required is stable as the mesh is
  refined, while in the nonsmooth case number iterations increase.}
\end{figure}
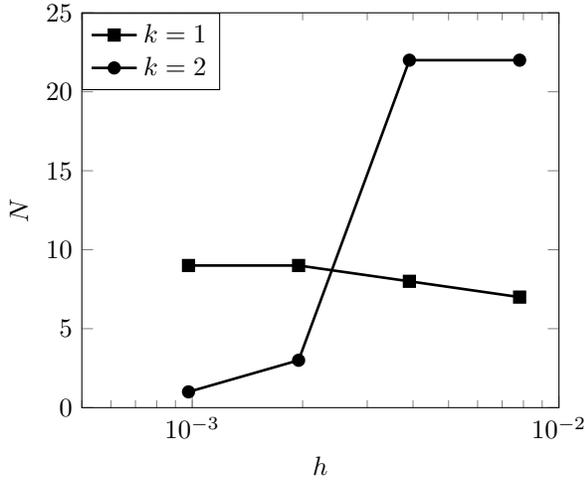
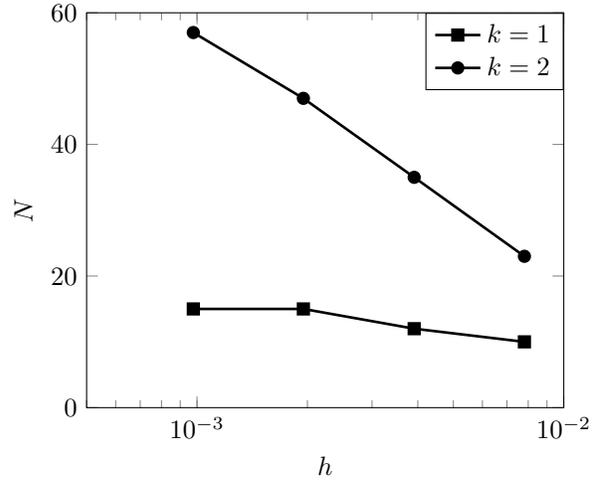

\subsection{Example 2: linear reaction with discontinuous boundary data}
\label{sec:example_2}

We now consider a more challenging variation on Example 1 where the boundary
condition is discontinuous:
\begin{equation}
  g_2(x, y) := \begin{cases}
    1 & \text{if $|x-\frac 1 2| < \frac{1}{\sqrt{5}}$},\\
          0 & \text{otherwise}.
   \end{cases}
\end{equation}
A solution can still be found using the method of characteristics, but this solution
inherits the discontinuity as it is propagated along the characteristics from the inflow
boundary. We will therefore not have the necessary regularity on the solution for error
estimates given in \ref{cor:rates_1} to hold, since the interpolation estimates are too
limited by the lack of regularity - indeed in light of Remark
\ref{rem:regularity_remarks}, we only expect the solution to lie in the graph space. The
particular challenge of this example is the spurious oscillations which are expected to
be exacerbated by the discontinuity. The exact solution is given by
\begin{equation}\label{eq:u_2}
  u_2(x, y)
  =
  \begin{cases}
    \exp\left( - \frac{y}{\sqrt{2}}\right) 
    & \text{if  $-\sqrt{\frac 2 5} 
    < y - x \sqrt{2} + \frac{1}{\sqrt 2}
    < \sqrt{\frac 2 5} $},\\
          0 & \text{otherwise}.
   \end{cases}
\end{equation}
In Figure \ref{fig:firedrake_rates_2} we display convergence rates for the
bound-preserving finite element method, and in Figure \ref{fig:ex_2_x_sections} a
comparison with the standard SUPG case without enforcing bounds is presented. As
expected, we do not obtain the desired convergence rate, but a stable solution is
obtained, with much reduced numerical artefacts, as seen in
\autoref{fig:ex_2_x_sections}.

\begin{figure}[htp!]
  \centering
  \begin{subfigure}{.4\linewidth}
      \includegraphics[width=\textwidth]{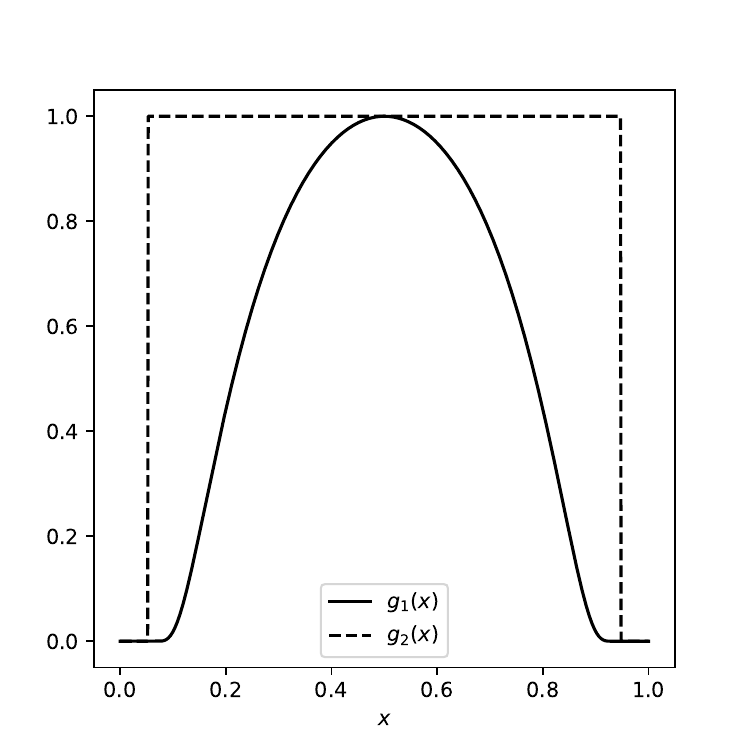}
      \caption{An illustration of the boundary conditions $g_1$ and $g_2$ chosen for 
      Examples 1 \& 2.}
  \end{subfigure}
  \hfil
  \begin{subfigure}{.5\linewidth}
      \includegraphics[width=\textwidth]{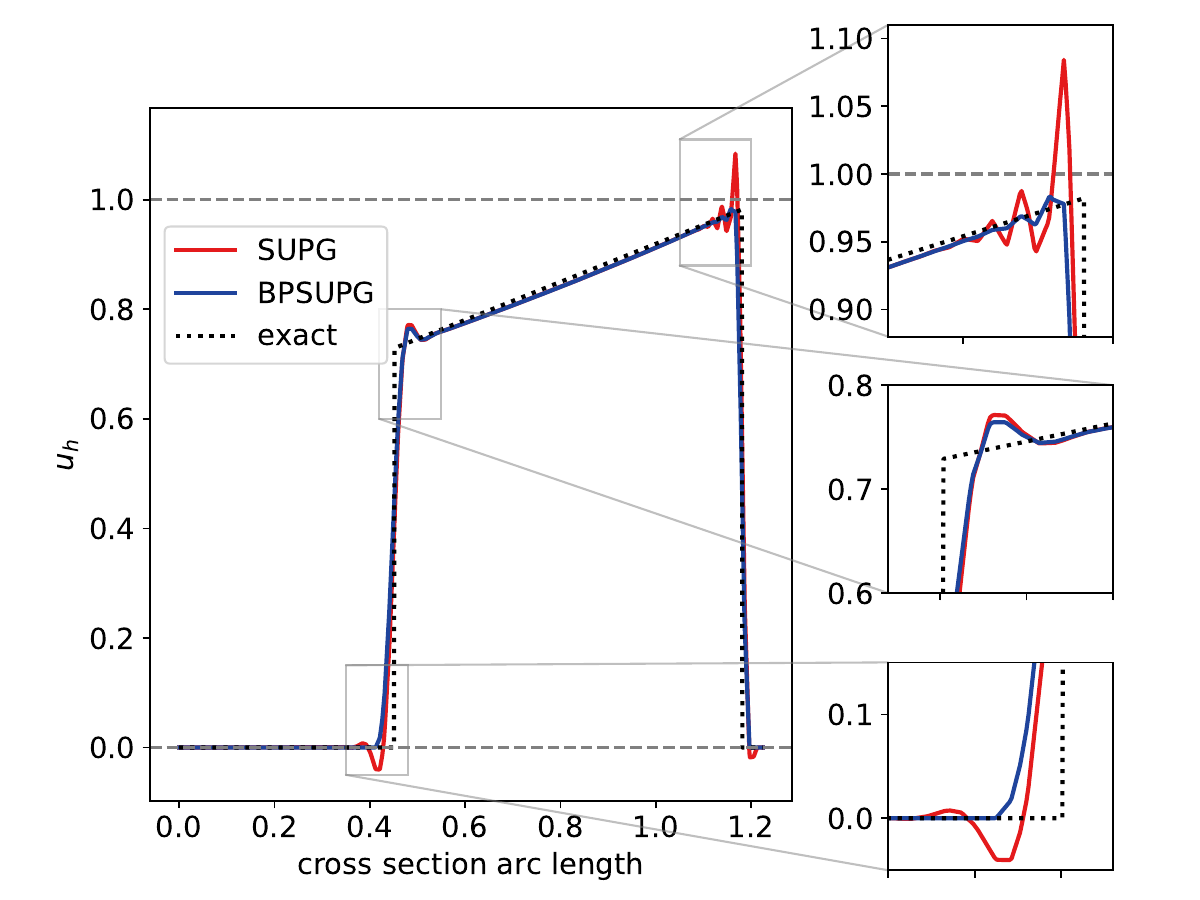}
      \caption{\label{fig:ex_2_x_sections} Cross section of exact and numerical
      solutions along the line perpendicular to the wind field \(\vec{b}_1\) and passing
      through the point \((1, 0)\).}
  \end{subfigure}
  \caption{Note that the standard SUPG solution exhibits oscillations which cause the
   violation of the analytical bounds on the solution, and that this issue is 
   remedied by the use of the bound-preserving method.}
\end{figure}

\subsection{Example 3: A discontinuous pure advection problem}
\label{sec:interior_layer}

To demonstrate the effectiveness of the bound preserving method for a problem with low
regularity and where violation of the maximum principle is expected when standard finite
element methods are used, we consider a pure advection problem with discontinuous data
on the inflow boundary. Numerical solutions to such problems typically suffer from
spurious under- and over-shoots as the discontinuity is propagated by the velocity
field. The advection field is chosen to be 
\begin{equation}\label{eq:3rd_wind}
  \vec b_3(x, y) 
  :=
  \frac{1}{\sqrt{x^2 +y^2}}(-y, x),
\end{equation}
so that the inflow boundary is in this case defined by
\begin{equation*}
  \Gamma^+ 
  = 
  \{(x, y) \in \partial \W : x = 1\} 
  \cup 
  \{(x, y) \in \partial \W : y =0\}.
\end{equation*}
Boundary data is prescribed as follows:
\begin{equation}\label{eq:3rd_bc}
  g_3(x, y) := \begin{cases}
    0 & \text{if $x < \frac{1}{3}$},\\
    \frac 1 2 & \text{if $\frac{1}{3} \leq x < \frac{2}{3}$},\\
    1 & \text{otherwise},
   \end{cases}
\end{equation}
resulting in a discontinuous piecewise constant solution having $\operatorname{H}_-(\W)$
regularity but no better, illustrated in \autoref{fig:contour_layered}. Numerical
results  are shown in \autoref{fig:rotational_wind}, which demonstrate both the strength
and weakness of this method. Indeed, the bound-preserving SUPG method (BPSUPG) satisfies
the bound $u_h(\vec x) \in [0,1]\,\, \forall \vec x \in \W$, and the oscillations that
occur near $u_h = 0$ and $u_h = 1$ are completely removed (see
\autoref{fig:rotation_with_zooms}), but oscillations which lie away from the minimum and
maximum values are not alleviated by the method.

\begin{figure}
  \centering 
\subcaptionbox{\label{fig:contour_layered} Visualisation of the piecewise constant exact
solution to Example 3. }
{\includegraphics[width=.48\linewidth]{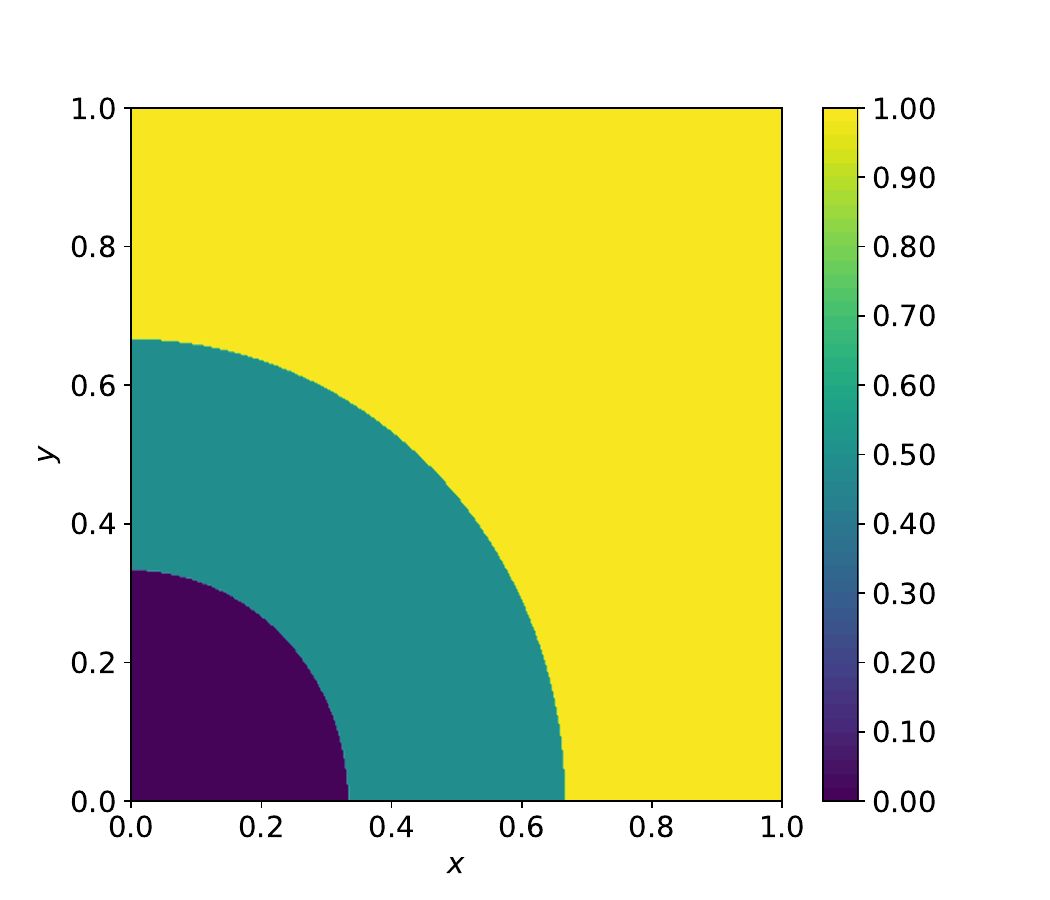}}
\hfill
\subcaptionbox{\label{fig:rotation_with_zooms} Cross section of exact and numerical
solutions along the line $y = x$.
}{\includegraphics[width=.48\linewidth]{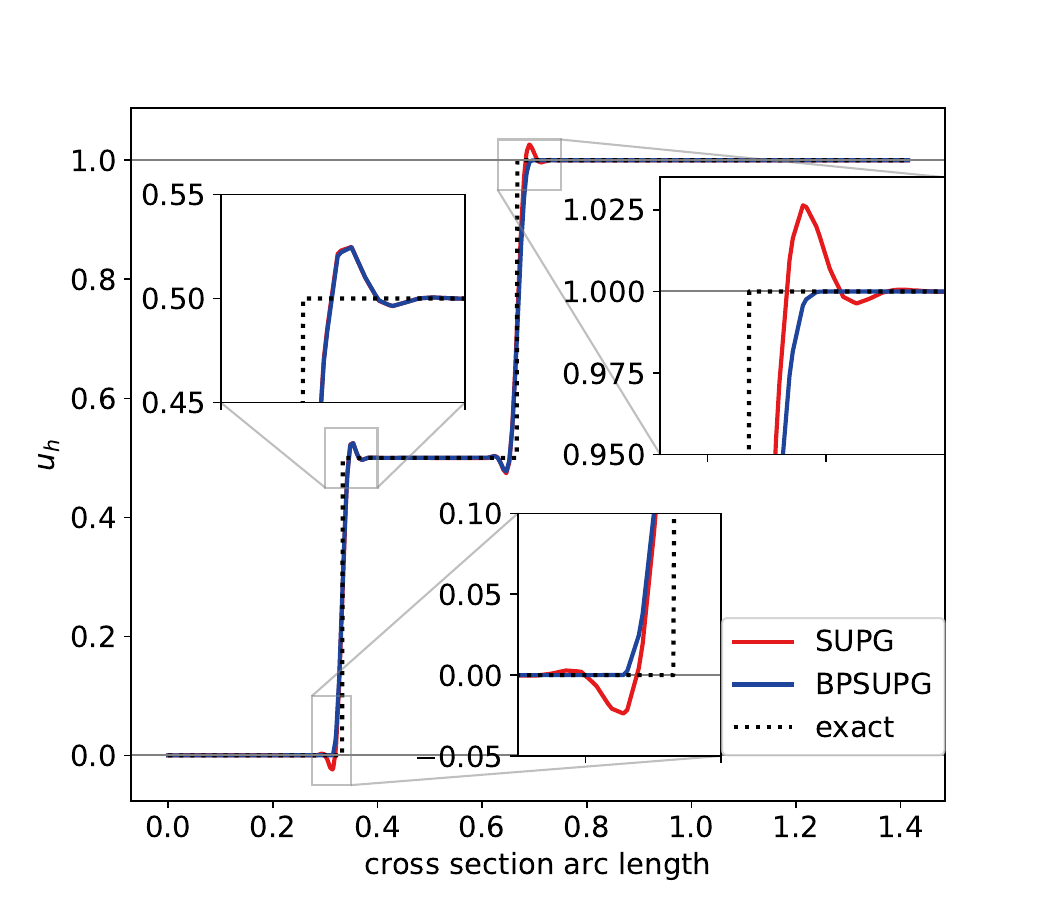}}
\caption{\label{fig:rotational_wind} Visualisations of exact and numerical solutions to
Example 3, \S\ref{sec:interior_layer}, The numerical solution obtained from the
variational inequality problem and discretised with BPSUPG using $\mathbb{P}_1$
elements, $h = 2^{-7}$, is shown in blue, while the regular SUPG finite element solution
with the same discretisation parameters is shown in red. The dotted line is the exact
solution, and grey horizontal lines indicate bounds on the exact solution which are
known a priori.}
\end{figure}

\subsection{Example 4: nonlinear reaction with smooth boundary data}
\label{sec:example_4}

We now consider the following problem
\begin{equation}\label{eq:nonlinear_pde_2}
  \begin{split}
  \vec b \cdot \grad u + c |u|^{-\frac 1 2}u &= 0 \quad \text{in}\,\, \W\\
  u &= g \quad \text{on}\,\, \Gamma_-,
  \end{split}
\end{equation}
which we can consider as a re-scaling of Equation \eqref{eq:nonlinear_pde}. We note that
an exact solution for this problem can be found using the method of characteristics as
long as $\vec b$ is sufficently well-behaved. Indeed, assuming that the boundary data is
non-negative, we consider the behaviour of $u$ along a characteristic which originates
at $\vec x_0$ on $\Gamma_-$. Then if the characteristic containing $\vec x_0$ is
parameterised by $t$, we have 
\begin{equation}\label{eq:characteristic_ode}
  \frac{\diff}{\diff t} u(\vec x(t)) + c |u|^{-\frac 1 2}u = 0.
\end{equation}
This ordinary differential equation is well-posed for all initial conditions $u_0$,
with solution given by (see also \autoref{fig:char_sol})
\begin{equation}\label{eq:exact_solution_4}
  u(\vec x(t))
  =
  \operatorname{sgn} (u_0)
  \left(\max\left\{0, \sqrt{|u_0|} - \frac{ct}{2}\right\}\right)^2.
\end{equation}
Thus, the solution approaches zero quadratically for $t < 2 \sqrt{|u_0|} \slash c$ and
is zero for all $ \geq 2 \sqrt{|u_0|} \slash c$.
\begin{figure}
  \centering
\subcaptionbox{\label{fig:char_sol}Solution to
\autoref{eq:characteristic_ode} with $u_0 = 1 \slash 2$ and $c=4$.}
{\includegraphics[width=.48\linewidth]{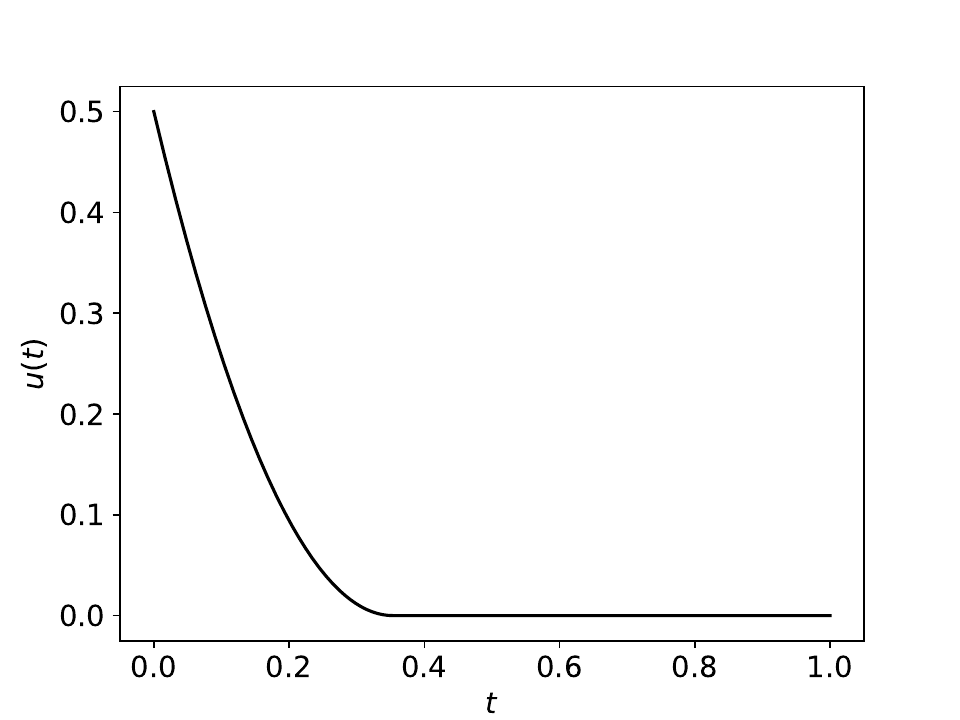}}
\hfill
\subcaptionbox{\label{fig:char_surf} Contour plot of exact solution for Example
4.}{\includegraphics[width=.48\linewidth]{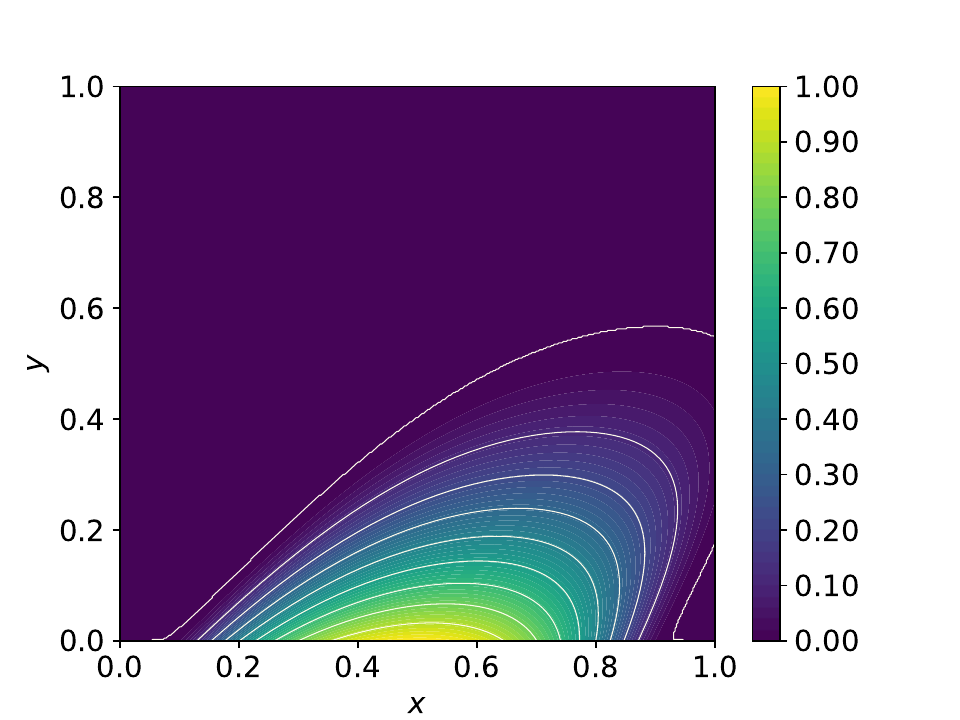}}
\caption{\label{fig:nonlinear_solutions}Visualisations of the exact solution to Example
4, given by \autoref{eq:u_3}.}
\end{figure}

Along a characteristic, the exact solution is quadratic until it reaches zero in finite
time given by $2 \sqrt{|u_0|}\slash c$ and remains zero thereafter. It is therefore
continuously differentiable, but with a second (weak) derivative which is discontinuous.
If $u_0 = 0$, the solution is trivial. Following these observations, we deduce that if
the boundary data is bounded in $[0,1]$, then the solution $u$ of
\autoref{eq:nonlinear_pde_2} satisfies the same bounds.

The inflow boundary data is set to be $g_1$, with $\vec b_1$ again chosen for the
advection field. It follows from the computation of the characteristics defined by $\vec
b$ together with the \autoref{eq:exact_solution_4} that the exact solution to this
problem is given by

\begin{equation}\label{eq:u_3}
  u_3(x, y)
  =
  \begin{cases}
    \left(\max\left\{0, \sqrt{g_1\left(x - \frac{y}{\sqrt{2}}, y\right)}
    - c \frac{y}{2\sqrt{2}}\right\}\right)^2
    & \text{if  $-\sqrt{\frac 2 5} 
    < y - x \sqrt{2} + \frac{1}{\sqrt 2}
    < \sqrt{\frac 2 5} $},\\
          0 & \text{otherwise}.
   \end{cases}
\end{equation} 
Since $g_1$ is smooth, the above observations imply that $u_3 \in \sobh{2}(\W)$. A
visualisation of the exact solution given in \autoref{eq:u_3} is provided in
\autoref{fig:char_surf}.

The reduced-space active set method is again used to solve the discrete variational
inequality, but this time an approximate Jacobian is used in the solve step, Equation
\eqref{eq:solve_step}, to avoid the singularity in the reaction coefficient. Since the
Jacobian is singular around $u = 0$, a small amount of regularisation was added and
observed to improve solver performance.

In \autoref{fig:nonlinear_firedrake_rates_1} we display convergence rates for the
bound-preserving finite element method. As long as the SUPG parameters are chosen to be
of the appropriate order, $\mathcal{O}(h^{3\slash 2})$ convergence is obtained.

\subsection{Example 5: nonlinear reaction with discontinuous boundary data}
\label{sec:example_5}

We test the nonlinear case with boundary data $g_2$ and advection field $\vec b_1$. As
above, an exact solution can be derived for this problem, however with discontinuous
boundary condition, the resulting solution has minimal regularity. Convergence rates are
shown in \autoref{fig:nonlinear_firedrake_rates_2}, again affected by the lack of
regularity of the exact solution in this case. A comparison of exact and numerical
solutions is given in \autoref{fig:discontinuous_nonlinear_zooms}

\begin{figure}
  \centering
  \includegraphics[width=0.7\textwidth]{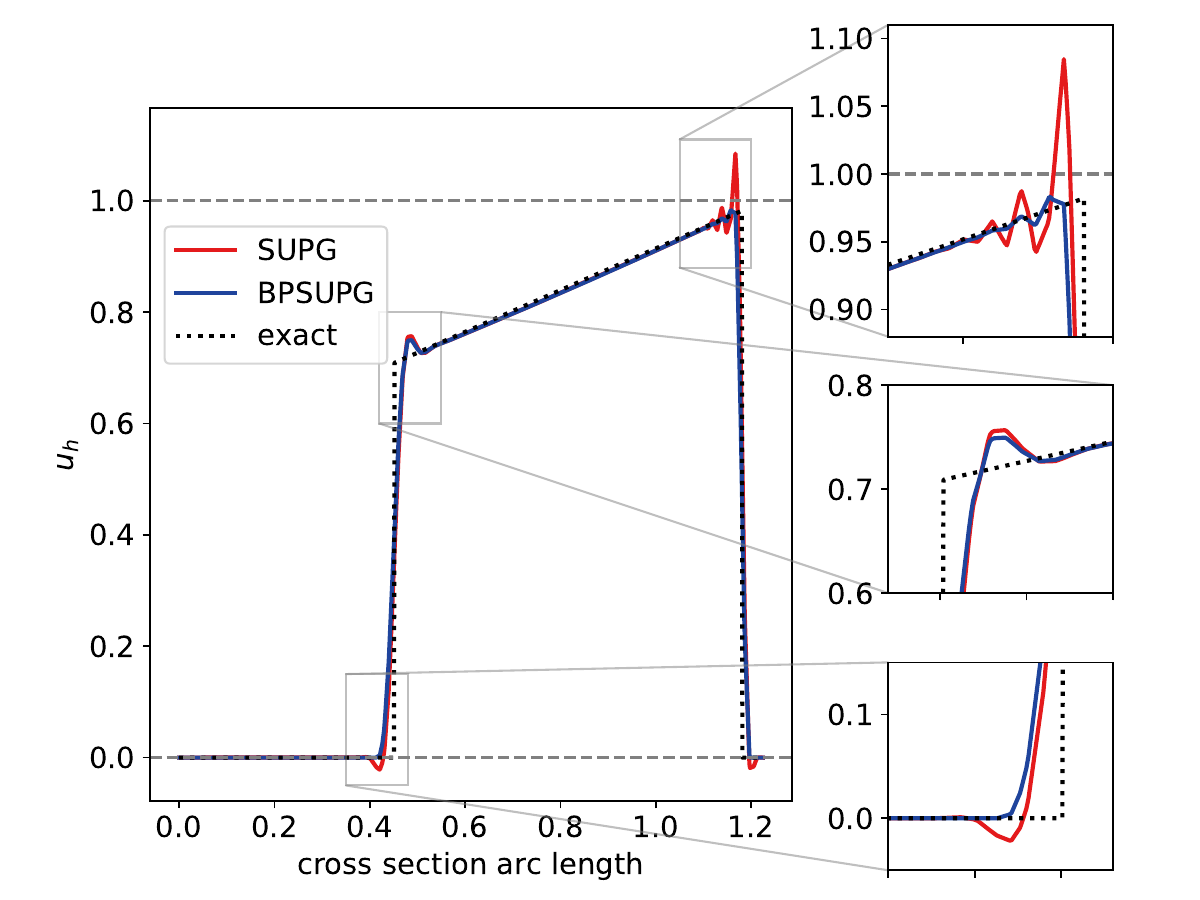}
  \caption{\label{fig:discontinuous_nonlinear_zooms}Cross section of exact and numerical
  solutions along the line perpendicular to the wind field \(\vec{b}_1\) and passing
  through the point \((1, 0)\).}
\end{figure}

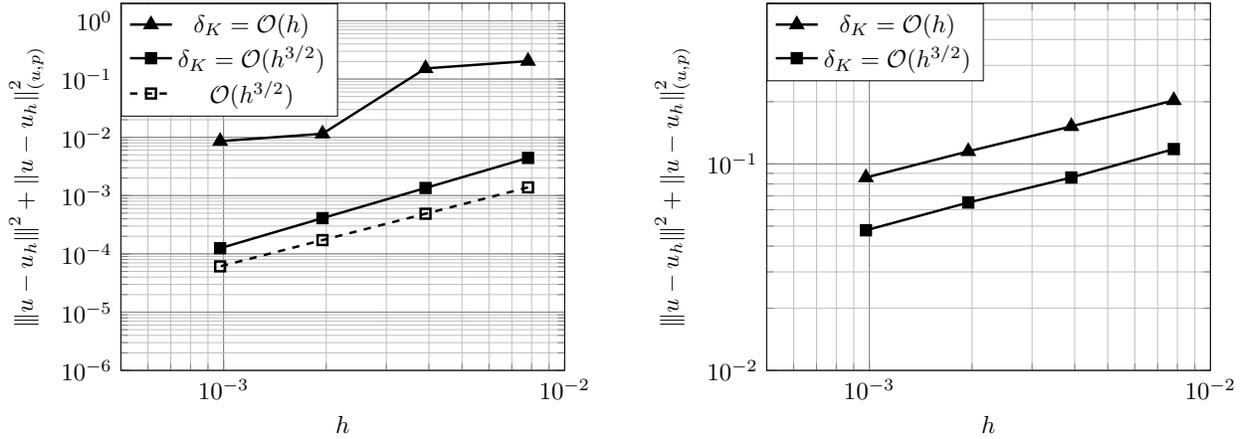
\begin{figure}
  \begin{subfigure}{.48\linewidth}
\begin{tikzpicture}[scale=0.93]
  \begin{axis}[
      width =\linewidth,
      xmode=log, ymode=log,
      xmin=5e-4, xmax=1e-2,
      ymin=1e-6, ymax=2,
      grid=both,
      major grid style={black!50},
      xlabel = \(h\),
      ylabel = \( \tripleNorm{u-u_h}^2
      +
      \Norm{u-u_h}^2_{(u,p)}\),
      legend style={at={(0.0,1)},anchor=north west}
  ]

  \addplot[solid, mark=triangle*, mark options={scale=1.4,solid}, color={black!100}, 
           line width=1.0] coordinates {
    (7.8125e-3, 2.03e-01 )
    (3.91e-3, 1.52e-01)
    (1.95e-3, 1.15e-02)
    (9.77e-4, 8.59e-03)};
    \addlegendentry{\(\delta_K = \mathcal{O}(h)\)}

  \addplot[solid, mark=square*, mark options={scale=1, solid}, color={black!100},
           line width=1.0] coordinates {
    (7.8125e-3, 4.43e-03)
    (3.91e-3, 1.35e-03)
    (1.95e-3, 4.10e-04)
    (9.77e-4, 1.25e-04 )};
  \addlegendentry{\(\delta_K = \mathcal{O}(h^{3\slash 2})\)}

    \addplot[dashed, mark = square*,
             mark options={scale=1,solid, fill = white, fill opacity=0},
             color={black!100}, line width=1.0] coordinates {
      (7.8125e-3, 2 * 6.91e-4)
      (3.91e-3, 2 * 2.44e-4)
      (1.95e-3, 2 * 8.61e-5)
      (9.77e-4, 2 * 3.05e-5)};
      \addlegendentry{\(\mathcal{O}(h^{3 \slash 2})\)}
  

  \end{axis}
  \end{tikzpicture}
  \caption{\label{fig:nonlinear_firedrake_rates_1} Approximation error for Example 4
  \S\ref{sec:example_4}, solution in \(\sobh{2}(\W)\).}
\end{subfigure}
\hfill
\begin{subfigure}{.48\linewidth}
  \begin{tikzpicture}[scale=0.93]
    \begin{axis}[
        width = \linewidth,
        xmode=log, ymode=log,
        xmin=5e-4, xmax=1e-2,
        ymin=1e-2, ymax=0.6,
        grid=both,
        major grid style={black!50},
        xlabel = \(h\),
        ylabel = \( \tripleNorm{u-u_h}^2
      +
      \Norm{u-u_h}^2_{(u,p)}\),
        legend style={at={(0.0,1)},anchor=north west}
    ]
    \addplot[solid, mark=triangle*, mark options={scale=1.4,solid}, color={black!100}, 
    line width=1.0] coordinates {
    (7.8125e-3, 2.03e-01 )
    (3.91e-3, 1.52e-01)
    (1.95e-3, 1.15e-01)
    (9.77e-4, 8.59e-02)};
    \addlegendentry{\(\delta_K = \mathcal{O}(h)\)}
    \addplot[solid, mark=square*, mark options={scale=1, solid}, color={black!100}, 
             line width=1.0] coordinates {
      (7.8125e-3, 1.18e-01)
      (3.91e-3, 8.59e-02)
      (1.95e-3, 6.50e-02)
      (9.77e-4, 4.77e-02  )};
    \addlegendentry{\(\delta_K = \mathcal{O}(h^{3\slash 2})\)}


    \end{axis}
    \end{tikzpicture}
    \caption{\label{fig:nonlinear_firedrake_rates_2} Approximation error for Example 5
    \S\ref{sec:example_5}, solution in \(u \in \operatorname{H}_-(\W) \backslash
    \sobh{1}(\W)\) only.}
  \end{subfigure}
  \caption{\label{fig:nonlinear_example_rates}Approximation errors for the
  bound-preserving finite element method in the full SUPG $+$ quasi norm.  The choice of
  SUPG parameters greatly affects the performance of the method. }
\end{figure}

\section{Conclusions and Further Work}

In this work, we introduced a method for preserving bounds on
numerical solutions to partial differential equations by reformulating
the standard discrete weak form as a discrete variational
inequality. The finite element method was analysed for both a linear
first-order problem and a first-order problem with a nonlinear
reaction term. The SUPG method was adapted to this bound-preserving
framework, and the expected convergence rates were successfully
retained.

While certain aspects of the analysis are tailored to the specific
problems considered, the proposed framework is broadly applicable and
can, in principle, be extended to any problem admitting a variational
formulation. This flexibility enables the application of the method to
a wide range of partial differential equations using existing finite
element techniques.

The approach outlined here shows particular promise for problems
involving singular coefficients or where strict preservation of bounds
is essential to ensure physical validity. Future work will explore the
extension of this framework to kinetic equations
\cite{AshbyChronholmHajnalLukyanovMacKenziePimPryer:2024,pim2024optimal},
where ensuring positivity of solutions is important physically, and to
models of non-Newtonian fluids \cite{ashby2025discretisation}, which
feature complex nonlinearities and require positive preserving schemes
to ensure well-posedness. Key to all these applications is the
efficient solution and computational scalability of this class of
method.

\newcommand{\etalchar}[1]{$^{#1}$}

\end{document}